\documentclass[12pt]{article}
\usepackage{amssymb, amsmath, amsthm, amscd}
\usepackage[dvips]{graphics}
\usepackage[utf8]{inputenc}
\usepackage[all,cmtip]{xy}
\usepackage{bbm}
\usepackage{enumitem}
\usepackage{setspace}
\usepackage[colorlinks=true,
            linkcolor=blue,
            urlcolor=blue,
            citecolor=blue]{hyperref}

\begin{document}

\newtheorem{lem}{Lemma}[section]
\newtheorem{pro}[lem]{Proposition}
\newtheorem{defi}[lem]{Definition}
\newtheorem{def/not}[lem]{Definition/Notations}
\newtheorem{thm}[lem]{Theorem}
\newtheorem{ques}[lem]{Question}
\newtheorem{cor}[lem]{Corollary}
\newtheorem{rem}[lem]{Remark}
\newtheorem{rqe}[lem]{Remarks}
\newtheorem{exa}[lem]{Example}
\newtheorem{exas}[lem]{Examples}
\newtheorem{obs}[lem]{Observation}
\newtheorem{corcor}[lem]{Corollary of the corollary}
\newtheorem*{ackn}{Acknowledgements}

\newcommand{\C}{\mathbb{C}}
\newcommand{\R}{\mathbb{R}}
\newcommand{\N}{\mathbb{N}}
\newcommand{\Z}{\mathbb{Z}}
\newcommand{\Q}{\mathbb{Q}}
\newcommand{\Proj}{\mathbb{P}}
\newcommand{\Rc}{\mathcal{R}}
\newcommand{\Oc}{\mathcal{O}}
\newcommand{\diff}{\textit{diff}}

\begin{center}

{\Large\bf  A  variational approach to the quaternionic Hessian equation}

\end{center}
\begin{center}
{\large Hichame Amal \footnote{Department of mathematics, Laboratory LaREAMI, Regional Center of trades of education and training, Kenitra Morocco,  hichameamal@hotmail.com},
 Sa\"{\i}d Asserda \footnote{Ibn tofail university, faculty of sciences, department of mathematics, PO 242 Kenitra Morroco, said.asserda@uit.ac.ma},
 Mohamed Barloub\footnote{Ibn tofail university, faculty of sciences, department of mathematics, PO 242 Kenitra Morroco, mohamed.barloub@uit.ac.ma}
}
\end{center}
\noindent{\small{\bf Abstract.}
In this paper, we introduce finite energy classes of quaternionic $m$-plurisubharmonic  functions of Cegrell type and define the quaternionic $m$-Hessian operator on some Cegrell's classes. We use the variational approach to solve the quaternionic $m$-Hessian equation when the right-hand side is a positive Radon measure.

\noindent{\small{\bf Keywords.}
 Variational approach. Cegrell's class. Quaternionic $m$-subharmonic function.
  Quaternionic $m$-Hessian equation.

\noindent{\small{\bf Mathematics Subject Classification 35A15 . 32U15 . 32U40 . 32W20}
\vspace{1ex}
\section*{Introduction}\label{section:introduction}
 The quaternionic Monge-Amp\`ere
operator is defined as the Moore determinant of the quaternionic Hessian of a function $u$.
 In \cite{A1}, Alesker  proved that $(\Delta \varphi)^{n}=fdV$ is solvable  when $\Omega$ is a strictly pseudoconvex domain, $f\in C(\Omega),\; f\geq 0$ with continuous boundary data $\varphi\in C(\partial \Omega)$ and the solution is  a continuous plurisubharmonic function. For the smooth case, he proved in \cite{A1}
 a result on the existence and the uniqueness of a  smooth plurisubharmonic  solution  of  $(\Delta \varphi)^{n}=fdV$ when $\Omega$ is the euclidean ball in $\mathbb{H}^{n}$ and $f\in C^{\infty}(\Omega) ,\;f> 0,\; \varphi \in C^{\infty} (\partial \Omega)$.
Zhu extended this result in \cite{Zh}
 when $\Omega$ is a bounded strictly pseudoconvex domain in $\mathbb{H}^{n}$ provided the existence of a subsolution.
In \cite{A3}, Alesker defined the Baston operator  to express the quaternionic Monge-Amp\`{e}re operator on quaternionic  manifolds by using methods of complex geometry.
Motivated by this formula Wan and Wang in \cite{WW2} introduced  two first-order differential operators $d_{0}$ and $d_{1}$ which behaves similarly as $\partial,\; \overline{\partial}$ and $\partial\overline{\partial}$  in complex pluripotential theory, and write $\Delta=d_{0}d_{1}.$ Therefore the quaternionic Monge-Amp\`{e}re operator $(\Delta u)^{n}$ has a simpler  explicit expression,
on this observation, some authors established and developed the quaternionic versions of several results in complex pluripotential theory.  Wang and Zhang proved the maximality of locally bounded plurisubharmonic solution to the problem above  with smooth boundary when $\Omega$ is an open set of $\mathbb{H}^{n}$ with $f=0$ and $\varphi \in L_{loc}^{\infty}(\Omega)$. In \cite{SR}, Sroka solved
the Dirichlet problem when the right hand side $f$ is merely in $L^p$ for any $p>2$, which is the optimal bound.
 The H\"{o}lder regularity of the solution  was proved independently by Boukhari in \cite{Bo1} and by Kolodziej and Sroka in \cite{Ko},  when $\Omega$ is strongly pseudoconvex bounded domain in $\mathbb{H}^{n}$ with smooth boundary when $\varphi \in C^{1,1}(\partial\Omega),\;0\leq f \in L^{p}(\Omega)$ for $p> 2.$  In \cite{Bo1}, authors  study the
Dirichlet Problem for the quaternionic Monge-Amp\`ere operator for measures which does not charge pluripolar sets and in \cite{Wa2}, D.Wan  applied a variational method based on pluripotential theory to solve  the quaternionic Monge-Amp\`{e}re equation on hyperconvex domains in $\mathbb{H}^{n}$.
  
 \par The class $\mathcal{QSH}_m(\Omega)$ of $m$-subharmonic functions and the quaternionic $m$-Hessian
operator $(\Delta u)^{m}\wedge \beta^{n-m}$ in a domain $\Omega$ of $\mathbb{H}^{n}$ are introduced independently in \cite{Liu} and \cite{Ba} and  some facts about related pluripotential  are given. In this paper we continue the investigation of the pluripotential theory for complex Hessian equations on a bounded domains of $\mathbb{H}^{n}$. We then considering the
following
 quaternionic $m$-Hessian equation
\begin{equation}\label{eq1}
 (\Delta \varphi)^{m}\wedge \beta^{n-m}= \mu,\;\;\;\;1\leq m \leq n,
\end{equation}
 where  $\mu$ a positive Radon measure and $\beta=\frac{1}{8}\Delta(\Vert q \Vert^{2})$ is the standart K\"{a}hler form in $\mathbb{H}^{n}$.
  
 \par The main goal of the present paper is to use the variational method initiated in \cite{Br1} for the complex Monge–Ampère equation to solve equation (\ref{eq1}). The idea of this method is to discredit the functional whose Hessian equation considered is the Euler-Lagrange equation and to minimize it on a suitable compact set of quatrnionic $m$-sh functions. We then show that this minimum point is the desired solution. In order to solve this equation, we introduce finite energy classes of quaternionic $m$-subharmonic functions of Cegrell type and extend the domain of definition of quaternionic $m$-Hessian operator to some Cegrell's classes, the functions of which are not necessarily bounded.

The paper is organized as follows. In Section \ref{sec1}, we recall some basic facts about quaternionic
$m$-subharmonic functions and the quaternionic $m$-Hessian operators.
 In Section \ref{sec2}, inspired by \cite{C3,C4,Lu2}, we introduce and study  Cegrell's classes for the quaternionic $m$-subharmic functions in a domain in $\mathbb{H}^{n}$ which is the generalizations of Cegrell's classes for the complex $m$-subharmonic functions \cite{Lu2}, we prove that integration by parts is allowed  and  establish some inequalities including the energy estimate for the quaternionic $m$-Hessian operator on these classes.  In Section \ref{sec3}, we develop a variational approach inspired by \cite{Lu2,Wa2} to solve (\ref{eq1}) and prove our main result
 \begin{thm}\label{th00}
Let  $\mu$ be a positive Radon  measure in $\Omega,\;p\geq 1.$
 Then, we have $ (\Delta \varphi)^{m}\wedge \beta^{n-m}= \mu$ with $\varphi \in \mathcal{E}_{m}^{p}(\Omega)$ if and only if   $ \mathcal{E}_{m}^{p}(\Omega)\subset L^{p}(\Omega,\mu).$
\end{thm}
 \section{Prelimenaries}\label{sec1}
  The Baston operator $\Delta$  is the first operator of $0$-Cauchy-Fueter complex on quaternionic manifold:
$$ 0 \longrightarrow  C^{\infty}(\Omega, \mathbb{C})\overset \Delta \longrightarrow  C^{\infty}(\Omega, \wedge^{2}\mathbb{C}^{2n}) \overset D \longrightarrow  C^{\infty}(\Omega, \mathbb{C}^{2}\otimes \wedge^{3}\mathbb{C}^{2n})\longrightarrow  \cdots $$
Alesker defined the quaternionic Monge-Amp\`{e}re operator as the $n$-th power of this
operator  when the manifold is flat. 
Motivated by this formula Wan and Wang in \cite{WW2} introduced  two first-order differential operators $d_{0}$ and $d_{1}$ which behaves similarly as $\partial,\; \overline{\partial}$ and $\partial\overline{\partial}$  in complex pluripotential theory, and write $\Delta=d_{0}d_{1}.$ Therefore the quaternionic Monge-Amp\`{e}re operator $(\Delta u)^{n}$ has a simpler  explicit expression,
 First,we  use the well
known embedding of the quaternionic algebra $ \mathbb{H}$ into $ End(\mathbb{C}^{2})$ defined by:
$$\;\; x_{0}+ix_{1}+jx_{2}+kx_{3} \longrightarrow \begin{pmatrix}
x_{0}+ix_{1} &-x_{2}-ix_{3}  \\
x_{2}-ix_{3} & x_{0}-ix_{1}
\end{pmatrix},$$
 and the conjugate embedding
$$\begin{array}{c}
\;\; \tau : \mathbb{H}^{n} \cong \mathbb{R}^{4n} \hookrightarrow \mathbb{C}^{2n \times 2} \\
(q_{0},\cdots ,q_{n-1})\mapsto (z^{j\alpha})\in \mathbb{C}^{2n \times 2}
\end{array}$$
$q_{l}=x_{4l}+ix_{4l+1}+jx_{4l+2}+kx_{4l+3}\;,l=0,1,\cdots,n-1\;\;\alpha=0,1$ with
  \begin{equation}\label{1}
  \begin{pmatrix}

z^{(2l)0} & z^{(2l)1}  \\
z^{(2l+1)0} & z^{(2l+1)1}  \\
\end{pmatrix}:=
\begin{pmatrix}

x_{4l} - ix_{4l+1} & -x_{4l+2} + ix_{4l+3} \\
x_{4l+2} + ix_{4l+3} & x_{4l} + ix_{4l+1} \\

\end{pmatrix}.
\end{equation}
Pulling back to the quaternionic space $ \mathbb{H}^{n} \cong \mathbb{R}^{4n} $ by  (\ref{1}), we define on $\mathbb{R}^{4n}$ first-order
differential operators $ \bigtriangledown_{j\alpha} $ as follows:
\begin{equation}\label{4}
 \begin{pmatrix}
 
\bigtriangledown_{(2l)0} & \bigtriangledown_{(2l)1}  \\
\bigtriangledown_{(2l+1)0} & \bigtriangledown_{(2l+1)1}  \\

\end{pmatrix}:=
\begin{pmatrix}

\partial_{x_{4l}} + i\partial_{x_{4l+1}} & -\partial_{x_{4l+2}} - i\partial_{x_{4l+3}} \\
\partial_{x_{4l+2}} - i\partial_{x_{4l+3}} & \partial_{x_{4l}}- i\partial_{x_{4l+1}} \\

\end{pmatrix}
\end{equation}
 The Baston operator is given by the determinants of $(2 \times 2)$-submatrices above.
Let $\wedge^{2k}\mathbb{C}^{2n}$ be the complex exterior algebra generated by $\mathbb{C}^{2n} , 0 \leq k \leq n$.
 Fix a basis $\lbrace \omega^{0},\omega^{1}\cdots,\omega^{2n-1}\rbrace$ of $\mathbb{C}^{2n}$. Let $\Omega$ be a domain in $\mathbb{R}^{4n}$. We define $$d_{0},d_{1} : C_{0}^{\infty}(\Omega,\wedge^{p}\mathbb{C}^{2n}) \longrightarrow C_{0}^{\infty}(\Omega,\wedge^{p+1}\mathbb{C}^{2n})\;\;\mbox{ by}\;$$
$$d_{0}F :=\sum_{k,I}\bigtriangledown_{k0}f_{I}\omega^{k} \wedge\omega^{I}$$
$$d_{1}F :=\sum_{k,I}\bigtriangledown_{k1}f_{I}\omega^{k} \wedge\omega^{I}$$ and
$$\Delta F:=d_{0}d_{1}F$$ for $F=\sum_{I}f_{I}\omega^{I}\in C_{0}^{\infty}(\Omega,\wedge^{p}\mathbb{C}^{2n}),$
 where $ \omega^{I}:= \omega^{i_{1}} \wedge \ldots  \wedge \omega^{i_{p}} $ for the multi-index $I = (i_{1},\ldots,i_{p})$.
The operators $d_{0}$ and $ d_{1}$ depend on the choice of the coordinates $ x_{j}$’s and the basis $\lbrace \omega^{j}\rbrace$.

 It is known (cf.\cite{WW2}) that the second operator $D$ in the 0-Cauchy-Fueter complex can be written as $DF:=\left(
                                                          \begin{array}{c}
                                                            d_{0}F \\
                                                            d_{1}F \\
                                                          \end{array}
                                                        \right).$\\
   Although $d_{0},d_{1}$  are not exterior differential, their behavior is similar to exterior differential:
   \begin{lem}\label{lem0}
   $d_{0}d_{1}=-d_{1}d_{0}$,  $d_{0}^{2}=d_{1}^{2}=0$; for $F\in C^{\infty}_{0}(\Omega,\wedge^{p}\mathbb{C}^{2n}),$ $G\in C^{\infty}_{0}(\Omega,\wedge^{q}\mathbb{C}^{2n}),$ we have  \begin{eqnarray}\label{eq3}
   d_{\alpha}(F\wedge G)=d_{\alpha}F\wedge G+(-1)^{p}F\wedge d_{\alpha}G, \        \ \alpha=0,1, \ \ d_{0}\Delta=d_{1}\Delta=0
   \end{eqnarray}
   \end{lem}
   We say $F$ is closed if $d_{0}F=d_{1}F=0,$ ie, $DF=0.$ For $u_{1},u_{2},\ldots,u_{n}\in C^{2},$ $\Delta u_{1}\wedge\ldots\wedge\Delta u_{k}$ is closed, with $k=1,\ldots,n.$
   Moreover, it follows easily from (\ref{4}) that $ \Delta u_{1}\wedge\ldots\wedge\Delta u_{n}$ satisfies the following remarkable identities:
   $$\Delta u_{1}\wedge\ldots\wedge\Delta u_{n}=d_{0}(d_{1}u_{1}\wedge\Delta u_{2}\wedge\ldots\wedge\Delta u_{n})=-d_{1}(d_{0}u_{1}\wedge\Delta u_{2}\wedge\ldots\wedge\Delta u_{n})$$
$$=d_{0}d_{1}(u_{1}\wedge\Delta u_{2}\wedge\ldots\wedge\Delta u_{n})=\Delta(u_{1}\wedge\triangle u_{2}\wedge\ldots\wedge\Delta u_{n}).$$
To write down the explicit expression, we define for a function $u\in C^{2},$
$$\Delta_{ij}u:= \frac{1}{2}(\nabla_{i0}\nabla_{j1}u-\nabla_{i1}\nabla_{j0}u).$$
$2\Delta_{ij}$ is the determinant of $(2\times 2)$-matrix of $i$-th and $j$-th rows of (\ref{4}).
 Then we can write
 \begin{equation}\label{5}
 \Delta u=\sum_{i,j=0}^{2n-1}\Delta_{ij}u\omega^{i}\wedge\omega^{j},
 \end{equation}
and for $u_{1},\ldots,u_{n}\in C^{2}$
\begin{equation}
\begin{array}{ll}
\Delta u_{1}\wedge\ldots\wedge\Delta u_{n}&=\sum_{i_{1},j_{1},\ldots}\Delta_{i_{1}j_{1}}u_{1}\ldots\Delta_{i_{n}j_{n}}u_{n}\omega^{i_{1}}\wedge\omega^{j_{1}}\ldots\wedge\omega^{i_{n}}\wedge\omega^{j_{n}}\\
&=\sum_{i_{1},j_{1},\ldots}\delta_{01\ldots(2n-1)}^{i_{1},j_{1}\ldots i_{n}j_{n}}\Delta_{i_{1}j_{1}}u_{1}\ldots\Delta_{i_{n}j_{n}}u_{n}\Omega_{2n},\\
\end{array}
\end{equation}
 where $$\Omega_{2n}=\omega^{0}\wedge \omega^{1}\wedge\ldots\wedge\omega^{2n-1}$$
 and $\delta_{01\ldots(2n-1)}^{i_{1},j_{1}\ldots i_{n}j_{n}}$:= the sign of the permutation from $(i_{1},j_{1},\ldots,i_{n},j_{n})$ to $(0,1,\ldots,2n-1),$
 if
$\{i_{1},j_{1}...,i_{n},j_{n}\}=\{0,1,...,2n-1\};$ otherwise, $\delta_{01..(2n-1)}^{i_{1}j_{1}..i_{n}j_{n}}=0.$ Note that $ \Delta u_{1}\wedge\ldots\wedge\Delta u_{n}$ is symmetric with respect to the permutation of $u_{1},...,u_{n}.$ In particulier, when $u_{1}=...=u_{n}=u,\;  \Delta u_{1}\wedge\ldots\wedge\Delta u_{n}$ coincides with $(\Delta u)^{n}=\wedge^{n}\Delta u.$\\
We denote by $\Delta_{n}( u_{1},..., u_{n})$ the coefficient of the form $ \Delta u_{1}\wedge\ldots\wedge\Delta u_{n},$ ie, $\Delta u_{1}\wedge\ldots\wedge\Delta u_{n}=\Delta_{n}( u_{1},..., u_{n})\Omega_{2n}.$ Then $\Delta_{n}( u_{1},..., u_{n})$ coincides with the mixed Monge-Amp\`{e}re operator $\det( u_{1},..., u_{n})$ while $\Delta_{n}u$ coincides with the quaternionic Monge-Amp\`{e}re operator $\det(u).$ See  \cite[Appendix A]{Wa1}.\\
The notions of quaternionic closed positive forms and closed positive
currents has been defined and detailed in \cite{Wa1,WW2,WZ} .
\begin{lem}{(Stokes type formula,\cite[Lemma 3.2 ]{WW2})}\label{lem1}

Assume that $T$ is a smooth $(2n-1)$-form in $ \Omega$, and
$h$ is a smooth function with $h = 0$ on $\partial \Omega$. Then we have
$\int_{\Omega}hd_{\alpha}T=-\int_{\Omega}d_{\alpha}h\wedge T\;\;\;\;\alpha=0,1.$
\end{lem}
The theory of Bedford-Taylor \cite{Be3} in complex analysis can be generalized to the quaternionic case. Let $u$ be a locally bounded psh function and let $T$ be a closed positive $2k$-current. Define $$\Delta u\wedge T:=\Delta(uT),$$ i.e., $(\Delta u\wedge T)(\eta):=uT(\Delta\eta)$ for test form $\eta.$ $\Delta u\wedge T$ is  also a closed positive current. Inductively, for $u_{1},\ldots,u_{p}\in\mathcal{QPSH}(\Omega)\cap L_{loc}^{\infty}(\Omega),$ Wan and Wang in \cite{WW2} showed that
$$\Delta u_{1}\wedge\ldots\wedge\Delta u_{p}:=\Delta(u_{1}\Delta u_{2}\wedge\ldots\wedge\Delta u_{p})$$ is closed  positive $2p$-current. In particular, for $u_{1},\ldots,u_{n}\in 
\mathcal{QPSH}(\Omega)\cap L_{loc}^{\infty}(\Omega),$ $\Delta u_{1}\wedge\ldots\wedge\Delta u_{n}=\mu\Omega_{2n}$ for a well-defined positive Radon measure $\mu.$\\
For any test $(2n-2p)$-form $\psi$ on $\Omega,$ we have $\displaystyle{\int}_{\Omega}\Delta u_{1}\wedge\ldots\wedge\Delta u_{p}\wedge\psi=\displaystyle{\int}_{\Omega}u_{1}\Delta u_{2}\wedge\ldots\wedge\Delta u_{p}\wedge\Delta\psi,$
where $u_{1},\ldots,u_{p}\in \mathcal{QPSH}(\Omega)\cap L_{loc}^{\infty}(\Omega).$
Given a bounded plurisubharmonic function $u$ one can define the quaternionic Monge-Amp\`{e}re measure
$$ (\Delta u)^{n}=\Delta u\wedge\Delta u\wedge\ldots\wedge\Delta u.$$
This is a nonnegative Borel measure.

\subsection{The quaternionic $m$-subharmonic functions}
Let $\beta=\frac{1}{8}\Delta(\Vert q \Vert^{2})$ denotes the standard K\"{a}hler form in $\mathbb{H}^{n}$ and
Let $$\widehat{\Gamma}_{m}
  :=\lbrace \alpha \in \wedge_{\mathbb{R}}^{2}\mathbb{C}^{2n} \;\; / \alpha \wedge \beta^{n-1} \geq 0,\alpha^{2} \wedge \beta^{n-2} \geq 0,\ldots ,\alpha^{m} \wedge \beta^{n-m} \geq 0 \rbrace, $$
where $\wedge_{\mathbb{R}}^{2}\mathbb{C}^{2n}$ is the space of all real $2$-forms in quaternion analysis.
\begin{defi}
  A $(2n-2k)$-current $T$ $(k\leq m)$ is called $m$-positive if for $\alpha_{1},\ldots,\alpha_{k}\in \widehat{\Gamma}_{m}$, we have
\begin{equation}
\alpha_{1}\wedge \ldots \wedge\alpha_{k}\wedge T \geq 0. \label{2}
\end{equation}
 \end{defi}
   \begin{defi}
 A real valued function $u:\Omega\subset \mathbb{H}^{n}\rightarrow \mathbb{R}\cup \lbrace-\infty\rbrace$  is called $m$-subharmonic  if it is subharmonic and, for any
 $\alpha_{1},\cdots,\alpha_{m-1}$ in $\widehat{\Gamma}_{m}$
 \begin{equation}\label{eq10}
  \Delta u\wedge\alpha_{1}\wedge\cdots\wedge \alpha_{m-1}\wedge \beta ^{n-m} \geq0.
 \end{equation}
The inequality (\ref{eq10}) says that $\Delta u \wedge \beta^{n-m}$ is $m$-positive.\\
 The class of all quaternionic $m$-subhaharmonic functions in $\Omega$ is denoted by $\mathcal{QSH}_{m}(\Omega)$.

 \end{defi}
 \begin{pro}\cite[proposition 4.3]{Liu}\label{pro1}
 Let $\Omega$ be an open subset of $\mathbb{H}^{n}$.
 \begin{enumerate}
 \item[1)] If $\alpha,\;\beta$ are non-negative numbers and $u,v \in  \mathcal{QSH}_{m}(\Omega)$, then $\alpha u+ \beta v  \in \mathcal{QSH}_{m}(\Omega);$ and $\max\lbrace u,v\rbrace \in \mathcal{QSH}_{m}(\Omega).$
 \item[2)]If $\Omega$ is connected and $\lbrace u_{j}\rbrace \subset  \mathcal{QSH}_{m}(\Omega)$ is a decreasing sequence, then $u=\displaystyle{\lim_{j\longrightarrow \infty}}u_{j} \in \mathcal{QSH}_{m}(\Omega)$ or $u\equiv -\infty.$
 \item[3)] If $u  \in \mathcal{QSH}_{m}(\Omega)$ and $ \gamma:\mathbb{R}\longrightarrow \mathbb{R}$ is a convex increasing function, then $\gamma \circ u  \in \mathcal{QSH}_{m}(\Omega).$
 \item[4)] If $u  \in \mathcal{QSH}_{m}(\Omega)$, then the standard regularization $u*\rho_{\epsilon} \in \mathcal{QSH}_{m}(\Omega_{\epsilon})$, where $\Omega_{\epsilon}:=\lbrace z/dis(z,\partial\Omega)> \epsilon\rbrace \;\;0< \epsilon \ll 1$.
 \item[5)] If $\lbrace u_{j}\rbrace \subset  \mathcal{QSH}_{m}(\Omega)$ is locally uniformly bounded from above, then $(\sup_{j}u_{j})^{*}  \in \mathcal{QSH}_{m}(\Omega)$, where $v^{*}$ denotes the regularization of $v$.
 \item[6)] Let $ \omega$ be a non-empty proper open subset of $ \Omega,\; u  \in \mathcal{QSH}_{m}(\Omega),\;v  \in \mathcal{QSH}_{m}(\omega)$ and $\displaystyle{\limsup_{q\longrightarrow \varsigma}}v(q)\leq u(\varsigma)$ for each $\varsigma \in  \partial\omega \cap \Omega$, then
 \begin{equation*}
  W:=  \left \{
 \begin{array}{ll}
 \max\lbrace u,v\rbrace,\;\;\;in\;\; \omega \\
 u, \;\;\;\;\;\;\;\;\;in\;\; \Omega\setminus\omega \\
\end{array}
  \in \mathcal{QSH}_{m}(\Omega).\right.
   \end{equation*}
\end{enumerate}
 \end{pro}
 \begin{lem}\label{lem2}
 \begin{enumerate}
 \item[1)] 
Let $v^{0},\ldots,v^{k}\in \mathcal{QSH}_{m}(\Omega)\cap L_{loc}^{\infty}(\Omega)$, let $(v^{0})_{j},\ldots,(v^{k})_{j}$ be a decreasing sequences of $m$-subharmonic functions in $\Omega$ such that $\displaystyle{\lim _{j\rightarrow \infty}}v_{j}^{t}=v^{t}$ point-wisely in $\Omega$ for $t=0,\ldots,k$ . Then for $ k\leq m \leq n$, the currents $v_{j}^{0}\Delta v_{j}^{1}\wedge \ldots \wedge \Delta v_{j}^{k}\wedge \beta^{n-m}$ converge weakly to $ v^{0}\Delta v^{1}\wedge \ldots \wedge \Delta v^{k}\wedge \beta^{n-m}$ as $j$ tends to $\infty$.
\item[2)] Let $\lbrace u_{j}\rbrace_{j\in \mathbb{N}}$ be a locally uniformly bounded sequence in $\mathcal{QSH}_{m}(\Omega)\cap L_{loc}^{\infty}(\Omega)$ that increases to $u \in \mathcal{QSH}_{m}(\Omega)\cap L_{loc}^{\infty}
(\Omega)$ almost every where in $\Omega$ and let $v^{1},\ldots,v^{m} \in \mathcal{QSH}_{m}(\Omega)\cap L_{loc}^{\infty}(\Omega)$. Then the currents $u_{j}\Delta v^{1}\wedge \ldots \wedge \Delta v^{m}\wedge \beta^{n-m}$ converge to $u\Delta v^{1}\wedge \ldots \wedge \Delta v^{m}\wedge \beta^{n-m}$ as $j\rightarrow\infty.$
\item[3)]  Let $\lbrace u_{j}\rbrace_{j \in \mathbb{N}}$ be a sequence in $\mathcal{QSH}_{m}(\Omega)\cap L_{loc}^{\infty}(\Omega)
$ that increases to $u \in \mathcal{QSH}_{m}(\Omega) \cap L_{loc}^{\infty}(\Omega)$ almost everywhere in $\Omega$ (with respect to Lebesgue measure). Then the currents $(\Delta u_{j})^{m}\wedge \beta^{n-m}$ converge weakly to $(\Delta u)^{m}\wedge \beta^{n-m}$ as $j\rightarrow \infty.$
 \end{enumerate}
 \end{lem}
\begin{proof}
 See \cite{Ba} and \cite{Liu}.   
\end{proof}
 \begin{defi}
 Let $\Omega\subset \mathbb{H}^{n}$, and let $E$ be an open subset of $\Omega$, the quaternionic $m$-capacity of $E$ with respect to $\Omega$ is defined by:
\begin{equation}\label{18}
C_{m}(E)=C_{m}(E,\Omega):= \sup\Big\{\int_{E}(\Delta u)^{m}\wedge \beta^{n-m}\;\;:u\in \mathcal{QSH}_{m}(\Omega),-1\leq u\leq 0\Big \}.
\end{equation}
\end{defi}
As in the complex case, the following proposition can be proved.
\begin{pro}\label{pro2}
\begin{enumerate}
\item[1)] If $E_{1}\subseteq E_{2}$, then $C_{m}(E_{1})\subseteq C_{m}(E_{2}).$
\item[2)] If $E\subseteq \Omega_{1}\subseteq \Omega_{2}$, then $C_{m}(E,\Omega_{1})\geq C_{m}(E,\Omega_{2}).$
\item[3)] $C_{m}(\cup_{j=1}^{\infty}E_{j})\leq \sum_{j=1}^{\infty}C_{m}(E_{j}).$
\item[4)] If $E_{1}\subseteq E_{2}\subseteq \cdots \; are\; borel\; sets\; in\; \Omega\;,\; then\; C_{m}(\cup_{j}E_{j})=\displaystyle{\lim _{j\rightarrow \infty}}C_{m}(E_{j}).$
\end{enumerate}
\end{pro}
Recall also  the following lemmas (see \cite{Ba} and \cite{Liu}) :
\begin{lem}\label{lem21}{(Integration by parts)}
Suppose that $u, v, w_{1},\ldots,w_{m-1} \in \mathcal{QSH}_{m}(\Omega)\cap  L_{Loc}^{\infty}(\Omega).$
 If $\displaystyle{\lim_{q\longrightarrow \partial\Omega}}u(q)=\displaystyle{\lim_{q\longrightarrow \partial\Omega}}v(q)$, then
$$  \int_{\Omega} v \Delta u \wedge T = \int_{\Omega}u \Delta v \wedge T.$$
Where $T=\Delta w_{1}\wedge \ldots \wedge \Delta w_{m-1}\wedge \beta^{n-m}.$
\end{lem}
\begin{lem}{(Maximum principle)}
 For $u,v \in \mathcal{QSH}_{m}(\Omega)\cap L_{loc}^{\infty}(\Omega)$, we have
 $$\chi_{\lbrace u>v \rbrace}(\Delta \max \lbrace u,v \rbrace )^{m}\wedge \beta^{n-m}=\chi_{\lbrace u>v \rbrace}(\Delta u)^{m}\wedge \beta^{n-m} $$ where $\chi_{A}$ is the characteristic function of a set $A.$
\end{lem}
 \begin{lem}{(Comparison principle )}

 Let $\Omega$ be a bounded domain in $\mathbb{H}^{n}$, let  $u,v \in \mathcal{QSH}_{m}(\Omega)\cap L_{loc}^{\infty}(\Omega)$, if for any $\xi \in \partial \Omega$
 $$\displaystyle{\liminf_{q\longrightarrow \xi}}(u(q)-v(q))\geq0 .$$
  Then
    $$\displaystyle{\int}_{\lbrace u< v\rbrace}(\Delta v)^{m}\wedge \beta^{n-m} \leq \displaystyle{\int}_{\lbrace u< v\rbrace}(\Delta u)^{m}\wedge \beta^{n-m}.$$
 \end{lem}
 \subsection{ The relatively extremal $m$-subharmonic function}
 Let $\mathcal{QSH}_{m}^{-}(\Omega) $ be the subclass of negative functions in $\mathcal{QSH}_{m}(\Omega).$
 \begin{defi}
A set $\Omega\subset \mathbb{H}^{n}$  is said to be
a quaternionic $m$-hyperconvex  if it is open, bounded, connected and if there exists
$\varphi \in \mathcal{QSH}_{m}^{-}(\Omega) $ such that $ \lbrace q \in \Omega\;; \varphi(q) < -c\rbrace \subset\subset \Omega,\; \forall c > 0$.  A such 
function is called an exhaustion function for $\Omega$.
\end{defi}
\begin{defi}
Let be a subset $E\Subset \Omega$. The relatively $m$-extremal function is defined by
$$u_{m,E,\Omega}=\sup\lbrace u(q):\;u\in \mathcal{QSH}_{m}(\Omega),\;u\leq0,\;u\vert_{E}\leq-1\rbrace,\;\;q\in\Omega.$$
\end{defi}
Its upper semi-continuous regularization $u_{m,E,\Omega}^{*}\in \mathcal{QSH}_{m}(\Omega)$.
\begin{pro}\label{pro3}
The relatively $m$-extremal  function has the following properties:
\begin{enumerate}
\item[1)] If $ E_{1} \subset E_{2}\Subset\Omega$, then $u_{m,E_{2},\Omega} \leq u_{m,E_{1},\Omega},$
\item[2)] If $E \subset \Omega_{1} \subset \Omega_{2}$ ,then $u_{m,E,\Omega_{2}} \leq u_{m,E,\Omega_{1}},$
\item[2)] If $K_{j} \searrow K$, with $K_{j}$ is compact in $\Omega$, then $(\lim u_{m,K_{j},\Omega}^{*})^{\ast}=u_{m,K,\Omega}^{*}$.
\end{enumerate}
\end{pro}
\begin{proof}
  The first and the second statements are trivial.
  For the third one, let $ u \in \mathcal{QSH}_{m}(\Omega)$ such  that $u\leq 0,\; u\vert_{K}\leq-1 .$ For each $\epsilon > 0$, define the open set $U_{ \epsilon}:= \{ u- \epsilon < -1 \}$. Since $U_{ \epsilon }$ contains all the compacts $K_{j}$ with $j>>1.$
  If we let $j \rightarrow + \infty$, we obtain $(\lim u_{m,K_{j},\Omega}^{*})^{\ast} \geq u- \epsilon,\;\;\;\forall \epsilon >0,$
   for each $ u \in \mathcal{QSH}_{m}(\Omega),\;u\leq 0$ and $u \leq-1 $ on $K$. From which follows the result.
\end{proof}
\begin{lem}\label{lem3}
Let $0 < r <R$  and note $ a =\dfrac{2n}{m} >1$. The relatively  $m$-extremal fonction $u_{m,\overline{B}(r),B(R)}$
is given by
$$u(q) = \max \Big(\dfrac{R^{2-2a}- \Vert q \Vert^{2-2a}}{r^{2-2a}-R^{2-2a}} ,-1\Big)$$
\end{lem}
\begin{proof}
The function $u$ is continuous in $\mathbb{H}^{n},\; u\vert _{\overline{B}(r)}=-1$ and $u\vert_{\partial B(R)}=0$.
Fix $q^{0} \in B(R)\setminus \lbrace0 \rbrace$ and define $$v(q)=\dfrac{R^{2-2a}-\Vert q\Vert^{2-2a}}{r^{2-2a}-R^{2-2a}}.$$
 Calculating $\dfrac{\partial^{2}v}{\partial q_{l} \partial \overline{q}_{k}}$ at $q^{0}$, we get
$$\dfrac{\partial^{2}v(q^{0})}{\partial q_{l} \partial \overline{q}_{k}}=\frac{4(a-1) \Vert q^{0} \Vert^{-2a}}{r^{2-2a}-R^{2-2a}}\Big( 2\delta_{lk}-a \dfrac{q^{0}_{k}.\overline{q^{0}}_{l}}{\Vert q^{0} \Vert^{2}}\Big).$$
Since the eigenvalues of the matrix $A:=\Big(2\delta_{lk}-a \dfrac{q^{0}_{k}.\overline{q^{0}}_{l}}{\Vert q^{0} \Vert^{2}}\Big)$ are $\lambda=(2,\ldots,2,2-a)$, then
$$\begin{array}{ll}
\tilde{S}_{p}(A)&=S_{p}(\lambda(A))=\dfrac{2(n-1)!}{p!(n-1-p)!}+\dfrac{(2-a)(n-1)!}{(p-1)!(n-p)!}\\
&=\dfrac{2(n-1)!}{(p-1)!(n-p)!}(\dfrac{n}{p}-\dfrac{n}{m}).\\
\end{array}$$
Thus $\tilde{S}_{p}(A)>0,\;\;\forall p<m$ and $\tilde{S}_{m}(A)=0$. Therefore $u\in \mathcal{QSH}_{m}(B(R)\setminus \overline{B}(r))$. By the maximum principle,  the result follows.
\end{proof}
\begin{pro}\label{pro4}
If $\Omega$ is a quaternionic $m$-hyperconvex and  $E$ is relatively compact in $\Omega$, then
$$\displaystyle{\lim_{q\longrightarrow \omega}u_{m,E,\Omega}(q)=0},\;\;\forall \omega\in\partial\Omega.$$
\end{pro}
\begin{proof}
Let $\rho < 0$ be an exhaustion function of $\Omega$, then there exists a constant $C>0$ such that $C\rho<-1$ on $E$. Thus $C\rho<u_{m,E,\Omega}.$ It is clear that $\displaystyle{\lim_{q\longrightarrow \partial \Omega}\rho(q)=0}$. Hence we get the result.
\end{proof}
\begin{pro}\label{pro5}
Let $\Omega$ be a quaternionic $m$-hyperconvex  and a compact $K \Subset \Omega$ such that $ u_{m,K,\Omega}^{*}  \equiv-1$ on $K$. Then
 $u_{m,K,\Omega}$ is continuous in $\Omega$.
\end{pro}
\begin{proof}
Let $u=u_{m,K,\Omega}$, and $\rho$ be a defining function for $\Omega$ such that $\rho<-1$ on $K$. Then $\rho \leq u$ in $\Omega.$
It is enough to prove that for each $\epsilon>0$, there exists a continuous function $v$ in the defining family for $u$ such that $u-\epsilon\leq v \leq u $ in $\Omega$.
Take $\epsilon>0$, there exists $\alpha>0$ such that $u-\epsilon <\rho$ in $\Omega\setminus \Omega_{\alpha}$ and $K\subset \Omega_{\alpha}$, where
$$\Omega_{\alpha}=\lbrace q\in \Omega;\;dist(q,\partial\Omega)>\alpha\rbrace.$$
One can find $\delta>0$ such that $u \ast \chi_{\delta}- \epsilon< \rho$ on $\partial \Omega_{\alpha}$ and $u\ast \chi_{\delta}- \epsilon<-1$ on $K$.
Define
 \begin{equation*}
  v:=  \left \{
 \begin{array}{ll}
 \max\lbrace u\ast\chi_{\delta},\rho\rbrace,\;\;\;in\;\; \Omega_{\alpha} \\
 u, \;\;\;\;\;\;\;\;\;in\;\; \Omega\setminus\Omega_{\alpha} \\
\end{array}
 \right.
   \end{equation*}
 Then $v$ is a continuous function in the defining family for $u$, and thus $u-\epsilon\leq v\leq u$  in $\Omega$.
\end{proof}
\begin{pro}\label{pro6}
Let a domain $\Omega \Subset \mathbb{H}^{n}$ and $E \subset \Omega$. Then
$u_{m,E,\Omega}^{ \ast} \equiv 0$ if and only if there exists  $v \in \mathcal{QSH}_{m}(\Omega),\;\; v < 0 $ such that $E \subset \lbrace v = -\infty\rbrace$.
\end{pro}
\begin{proof}
Denote $u=u_{m,E,\Omega}$. If $v\in \mathcal{QSH}_{m}(\Omega),\; v< 0$ and $E\subset\lbrace v=-\infty\rbrace$, then $\forall \varepsilon>0,\;\varepsilon v\leq u$ in $\Omega$ and hence $u=0$ almost everywhere in $\Omega$. Thus $u^{\ast} =0.$

Suppose that $u^{\ast}=0$. Then there exists $a\in \Omega$ such that $u(a)=0$ because of $u^{\ast}=u$ almost everywhere. For each $k\in \mathbb{N}$, there exists $v_{k}\in \mathcal{QSH}_{m}(\Omega)$ such that $v_{k}<0 $ in $\Omega$, $v_{k}<-1$ in $E$ and $v_{k}(a)> -2^{-k}.$
Define $$v(q)= \displaystyle{ \sum_{k=1}^{\infty}v_{k}(q)}\;\;\;\;\;q\in \Omega.$$
We have $v(a)>-1,\;\;v<0$ in $\Omega$, $v = -\infty$ in $E$. Since $v$ is the limit of a quaternionic $m$-subharmonic decreasing sequence of its partial sums and $v\neq -\infty$, then $v\in \mathcal{QSH}_{m}(\Omega).$
\end{proof}
\begin{cor}\label{cor1}
Let $\Omega \subset \mathbb{H}^{n}$ and $E =\cup_{j} E_{j},$ where $j=1,2,\ldots \;,E_{j}\subset\Omega$.
 If $u_{ m,E_{j},\Omega}^{*} \equiv 0,\;\forall j,$
 then $u_{m,E,\Omega}^{*}\equiv 0$.
\end{cor}
\begin{proof}
From  proposition \ref{pro6} there exists $v_{j}\in \mathcal{QSH}_{m}(\Omega)$ such that $v_{j}<0$ and $v_{j}= -\infty$ in $E_{j}.$ One can Take a point $a\in \Omega$ such that $v_{j}(a)>-\infty, $
 $v_{j}(a)>-2^{-j},\;\;\forall j$,  then $v=\sum_{j}v_{j}\in \mathcal{QSH}_{m}^{-}(\Omega)$ and $v\equiv-\infty$ in $E$. Applying proposition \ref{pro6}, we obtain $u^{\ast}_{m,E,\Omega}=0.$

\end{proof}
\begin{pro}\label{pro7}
Let  $\Omega \subset \mathbb{H}^{n}$ be a quaternionic $m$-hyperconvex   and $K \subset \Omega$ be a compact that is union of a family of closed  balls. Then $ u_{m,K,\Omega}^{*}  = u_{m,K,\Omega}$ is continuous in $\Omega$. In  particular, if $K \subset \Omega$ is
compact such that $0 < \varepsilon < dist(K, \partial\Omega)$, then $u_{m,K_{\varepsilon},\Omega}$ is continuous, where $K_{\varepsilon}=\lbrace q\in\Omega/ dist(q,K)\leq \varepsilon\rbrace$.
\end{pro}
\begin{proof}
It suffices to prove that $u=u_{m,K,\Omega}$ is continuous in $\partial K$. Let $b\in \partial K$. One
can choose $a\in K$ and $0<r<R$ such that $b\in \overline{B}(a,r)\subset K$ and $B(a,R)\subset\Omega$.
If $q\in B(a,R)$, then $u(q)\leq u_{m,\overline{B}(a,r),\Omega}(q)\leq u_{m,\overline{B}(a,r),B(a,R)}(q).$ From lemma \ref{lem3}   follows $\displaystyle{\lim_{q\longrightarrow b}u(q)=-1.}$
To prove the second conclusion note that $K_{\epsilon}=\cup_{a\in K}\overline{B}(a,\epsilon).$
\end{proof}
The following result is a consequence of propositions \ref{pro3}  and \ref{pro7}.
\begin{cor}\label{cor2}
Let $ \Omega \subset \mathbb{H}^{n}$ be a quaternionic $m$-hyperconvex and  $K \subset \Omega$ be a compact. Then
$\displaystyle{\lim_{\varepsilon\longrightarrow 0}}u_{m,K_{\varepsilon},\Omega} = u_{m,K,\Omega}$.
In particular, $u_{m,K,\Omega}$ is lower semi-continuous.
\end{cor}
\begin{pro}\label{pro8}
If $ \Omega \subset \mathbb{H}^{n}$ be a quaternionic $m$-hyperconvex  and $K \subset \Omega$ be a compact, then
$u_{m,K,\Omega}^{*}$ is $m$-maximal in $\Omega \setminus K$.
\end{pro}
\begin{proof}
In view of corollary \ref{cor2}, proposition \ref{pro7} and lemma \ref{lem2}, we may assume that $u=u_{m,K,\Omega}$ is continuous. Fix $B=\overline{B}(a,r)\subset \Omega\setminus K$ and define
 \begin{equation*}
  v:=  \left \{
 \begin{array}{ll}
  \varphi,\;\;\;in\;\; B \\
 u, \;\;\;\;\;\;\;\;\;in\;\; \Omega\setminus B \\
\end{array}
 \right.
   \end{equation*}
   where $\varphi$ is the unique solution to the following Dirichlet problem (see \cite[Theorem 3.1]{Ba})
    \begin{equation*}
\left\{
     \begin{array}{lll}
     \varphi\in \mathcal{QSH}_{m}(B)\cap C(\overline{B})\\
            (\Delta \varphi)^m\wedge\beta^{n-m}=0,\;\;\;\mbox{in}\; B\\
        \varphi_{\vert_{ \partial \Omega}}=u.
     \end{array}
   \right.
\end{equation*}
Clearly, $v\in \mathcal{QSH}_{m}(\Omega),\;v\geq 0$ and $v<-1$ on $K$. Hence $v\leq u$ in $\Omega$, On the other hand, $\varphi\geq u$ in $B$. Therefore $u=\varphi$ in $B$, Since $B$ was chosen arbitrarily, we get the desired result.
\end{proof}
\begin{thm}\label{th1}
 Let $ \Omega \Subset \mathbb{H}^{n}$ be a quaternionic $m$-hyperconvex  and $K \Subset \Omega$ be a compact, then
$$C_{m}(K, \Omega) = \int_{\Omega}(\Delta u_{m,K,\Omega}^{\ast})^{m}\wedge \beta^{n-m}.$$

Moreover, if $u_{m,K,\Omega}^{\ast}>-1$ on $K$, then $C_{m}(K ,\Omega)=0.$
\end{thm}
\begin{proof} Let $\rho$ be an exhaustion function of  $\Omega$ such that $\rho<-1$ on $K$. Denote $u:=u_{m,K,\Omega} $ and  fix $\epsilon  \in (0 ,1)$ and $v\in \mathcal{QSH}_{m}(\Omega,(0,1-\epsilon)).$
In view of  proposition \ref{pro7} and corollary \ref{cor2}, we can find an increasing sequence $(u_{j})\in C(\Omega)\cap \mathcal{QSH}_{m}(\Omega,[-1,0])$ that converges to $u$. We may assume that $u_{j}\geq\rho$ on $\Omega$.
We have $$K\subset \lbrace u_{j}<v-1\rbrace\subset \lbrace \rho < v-1 \rbrace \subset\lbrace \rho < -\epsilon\rbrace.$$
by  the comparison  principle, we get
$$\int_{K}(\Delta v)^{m}\wedge \beta^{n-m}\leq \int_{\lbrace u_{j}<v-1\rbrace}(\Delta v)^{m}\wedge \beta^{n-m}\leq \int_{\lbrace u_{j}<v-1\rbrace}(\Delta u_{j})^{m}\wedge \beta^{n-m}.$$
 It follows from lemma \ref{lem2} that $(\Delta u_{j})^{m}\wedge \beta^{n-m}$ converges weakly to $(\Delta u^{\ast})^{m}\wedge \beta^{n-m}.$
Hence,$$\int_{K}(\Delta v)^{m}\wedge \beta^{n-m} \leq \int_{\overline{{\lbrace \rho <v- \epsilon \rbrace}}}(\Delta u^{\ast})^{m} \wedge \beta^{n-m}.$$
 Proposition \ref{pro8}  implies that $$\int_{K}(\Delta u^{\ast})^{m}\wedge \beta^{n-m}= \int_{\overline{{\lbrace \rho<v-\epsilon\rbrace}}}(\Delta u^{\ast})^{m}\wedge \beta^{n-m}.$$
Note that for each $\epsilon \in (0,1)$, we have
$$C_{m}(K,\Omega)=(1-\epsilon)^{-m}\sup \lbrace \int_{K}(\Delta v)^{m}\wedge \beta^{n-m}\;/\;v\in \mathcal{QSH}_{m}(\Omega,(0,1-\epsilon))\rbrace.$$  For the second statement,  one can suppose that $u^{\ast}>\epsilon -1$ on $K$. Therefore
$$C_m(K, \Omega) \geq
\int_{K}(\Delta u_{m,U_j,\Omega})^{m}\wedge \beta^{n-m}=\int_{K}(\Delta (\dfrac{u^{\ast}}{1-\epsilon}))^{m}\wedge \beta^{n-m}=(1-\epsilon)^{-m}C_{m}(K,\Omega),$$ which implies that $C_{m}(K,\Omega)=0.$

  \end{proof}
\begin{defi}
 Let $\Omega$ be an open set in $\mathbb{H}^{n}$, and let $\mathcal{U} \subset \mathcal{QSH}_{m}(\Omega)$ be a family of functions which is locally
bounded from above. Define
$$u(q) = \sup\lbrace v(q)\; /\; v \in \mathcal{U}\rbrace.$$
A set of the form $N = \lbrace q \in \Omega \;/\; u(q) < u^{\ast}(q)\rbrace$ and all their subsets are called $m$-negligible.
\end{defi}
\begin{defi}
 We say that  $ E \subset \mathbb{H}^{n}$ is  $m$-polar set, if for each $q \in E$ there exist a neighborhood $V$ of $q$ and $v \in \mathcal{QSH}_{m}(V)$ such that $E \cap V \subset \lbrace v=- \infty\rbrace$.
 If  $E \subset \lbrace v = - \infty \rbrace$ with $v \in \mathcal{QSH}_{m}(\mathbb{H}^{n})$, we say that $E$
 is globally $m$-polar.
\end{defi}
\begin{pro}\label{pro9}
 Let  $\Omega$ be  an open  in $\mathbb{H}^{n}$ and $u \in \mathcal{QSH}_{m}(\Omega)\cap L_{loc}^{\infty}(\Omega).$   Then for each $m$-polar set $E \subset \Omega$ we have
$$\int_{E}u(\Delta u)^{m}\wedge \beta^{n-m}=0. $$
\end{pro}
\begin{proof}
We cover $E$ by a family of closed balls $B_{j}=B(b_j;r_j)$ such that $E_j=E\cap B_j \subset \{v_j=-\infty \},$ where $v_j\in \mathcal{QSH}_{m}(B_j).$
It is sufficient to prove that $$\forall j,\;\;\int_{E_j}u(\Delta u)^{m}\wedge \beta^{n-m}=0. $$
Fix $j\in \mathbb{N}$ and  prove that $$\int_{A_r}u(\Delta u)^{m}\wedge \beta^{n-m}=0, $$
for each $r< r_j$ where $A_r=v_j^{-1}(-\infty)\cap B(b_j,r).$
We may assume that $v_j< 0$, from proposition \ref{pro6} follows $u_{m,K,B(b_j,r)}^\ast=0$ where $K\Subset A_r$ is a compact set.
Applying Theorem \ref{th1}, we obtain $C_m(K,B(b_j,r))=0$ which implies that $$\int_{K}u(\Delta u)^{m}\wedge \beta^{n-m}=0. $$
\end{proof}
\section{Cegrell’s classes $\mathcal{E}_{m}^{0}, \mathcal{E}_{m}, \mathcal{F}_{m}, \mathcal{E}_{m}^{p}, \mathcal{F}_{m}^{p}$ and approximation of $m$-sh functions}\label{sec2}
Throughout this section, we let $\Omega$ denote a quaternionic $m$-hyperconvex domain.
We define the Cegrell's classes of quaternionic $m$-subharmonic functions.
\begin{defi}\label{def1}
\begin{itemize}
\item We denote $\mathcal{E}^0_m(\Omega)$ the class of bounded functions that is belong to $\mathcal{QSH}_{m}^{-}(\Omega)$ such that
$\displaystyle\lim_{q\rightarrow \xi}u(q)=0$, $\forall\xi\in\partial\Omega$ and $\displaystyle\int_{\Omega}(\Delta u)^{m}\wedge \beta^{n-m}<+\infty.$
\item Let $u \in \mathcal{QSH}_{m}^{-}(\Omega)$, we say that $u$  belongs to $\mathcal{E}_m(\Omega)$ (shortly $\mathcal{E}_m$) if for each $q_0 \in \Omega$, there exists an open neighborhood $U \subset \Omega$ of $q_0$ and a decreasing sequence $(u_j) \subset \mathcal{E}_m^0(\Omega)$ such that
  $u_j \downarrow u$ on $U$ and $\displaystyle\sup_j\int_\Omega (\Delta u_j)^{m}\wedge \beta^{n-m}<+\infty.$
\item We denote by $\mathcal{F}_m(\Omega)$ (or $\mathcal{F}_m$) the class of functions $u \in \mathcal{QSH}_{m}^{-}(\Omega)$ such that there exists a sequence $(u_j) \subset \mathcal{E}_m^0(\Omega)$  decreasing to $u$ in $\Omega$ and $\displaystyle\sup_j\int_\Omega(\Delta u_j)^{m}\wedge \beta^{n-m} <+\infty.$
\item For every $p\geq 1$, $\mathcal{E}^p_m(\Omega)$ denote the class of functions $\psi \in \mathcal{QSH}^{-}_m(\Omega)$ such that there
exists a decreasing sequence  $(\psi_j)\subset\mathcal{E}^0_m(\Omega)$ such that $\displaystyle\lim_{j\rightarrow +\infty}\psi_j(q)=\psi(q),$ and  $\displaystyle\sup_j\displaystyle\int_{\Omega}(-\psi_{j})^p(\Delta\psi_j)^{m}\wedge \beta^{n-m}<+\infty.$\\
\;\;\; If moreover $\displaystyle\sup_j\int_{\Omega}(\Delta\psi_j)^{m}\wedge \beta^{n-m}<+\infty$ then, by definition, $\psi\in \mathcal{F}^p_m(\Omega).$
\end{itemize}
\end{defi}
\begin{thm}\label{th3}
For any function $\varphi \in \mathcal{QSH}_{m}^{-}(\Omega)$, there is a decreasing sequence $(\varphi _{j})\in
\mathcal{QSH}_{m}(\Omega) $  satisfying the following properties:
\begin{enumerate}
\item[i)] $\varphi_{j}$ is continuous on $\overline{\Omega}$ and $\varphi_{j}\equiv 0$ on $\partial\Omega,$
\item[ii)] For each $j,\;\;
\displaystyle{\int}_{\Omega}(\Delta \varphi_j)^{m}\wedge \beta^{n-m}<+\infty,$
\item[iii)]$\lim_{j\longrightarrow+\infty}\varphi_{j}(q) = \varphi(q), $ for $ q \in \Omega.$

\end{enumerate}
\end{thm}
\begin{proof}
If $B$ is a closed ball in $\Omega$, then by Proposition \ref{pro7}, $u= u_{m,B,\Omega}$ is continuous in $\overline{\Omega}$ and $supp((\Delta u)^{m}\wedge \beta^{n-m})\Subset \Omega$.
 We can follow the lines in \cite[Theorem 2.1]{C4}.
\end{proof}
\begin{lem}\label{lem4}
$C_{0}^{\infty}(\Omega)\subset \mathcal{E}_{m}^{0}(\Omega)\cap C(\overline{\Omega})-\mathcal{E}_{m}^{0}(\Omega)\cap C(\overline{\Omega})$.
\end{lem}
\begin{proof} 
Fix $\chi\in C_{0}^{\infty}(\Omega)$ and choose $0 > \psi \in  \mathcal{E}_{m}^{0}(\Omega)$, choose $A$ so large that $\chi+A\vert q \vert^{2}\in \mathcal{QSH}_{m}(\Omega)$. Let $a,b\in \mathbb{R}$ such that
$$a<\inf \chi< \sup_{\Omega}(\vert \chi\vert +A\vert q \vert^{2})<b.$$
and define $$\varphi_{1}= \max(\chi +A\vert q\vert^{2}-b,B\psi)\;,\; \varphi_{2}=\max(A\vert q\vert^{2}-b,B\psi)$$ where $B$ is so large that $B \psi<a-b$ in $supp(\psi)$. Then $\chi=\varphi_{1}-\varphi_{2}$ and $\varphi_{1},\; \varphi_{2}\in \mathcal{E}_{m}^{0}(\Omega)$  by Proposition \ref{pro1}. 
\end{proof}
Let $u,v\in C^{2}(\Omega)$, define $$\gamma(u,v):=\dfrac{1}{2}(d_{0}u\wedge d_{1}v-d_{1}u\wedge d_{0}v):=\dfrac{1}{2}\sum_{i,j=0}^{2n-1}(\bigtriangledown_{i0}u\bigtriangledown_{j1}v-\bigtriangledown_{i1}u\bigtriangledown_{j0}v)\omega^{i}\wedge \omega^{j}.$$
In particular, $\displaystyle{\gamma(u,u)=d_{0}u\wedge d_{1}u=\dfrac{1}{2}\sum_{i,j=0}^{2n-1}(\bigtriangledown_{i0}u\bigtriangledown_{j1}u-\bigtriangledown_{i1}u\bigtriangledown_{j0}u)\omega^{i}\wedge \omega^{j}}.$
 Let $u, v, w_{1},\ldots , w_{k}$ be a locally bounded quaternionic $m$-subharmonic functions in $\Omega$, $ k+1\leq m\leq n.$
Then the following statements hold.
\begin{pro}\label{pro111}
 \begin{enumerate}
\item[1)] The mixed product $\gamma (u, v) \wedge \Delta w_{1} \wedge \ldots\wedge \Delta w_{k}\wedge\beta^{n-m}$ is well defined as a $(2n-2m+2k+2)$-current.
\item[2)] Let ${u_{j}}, {v_{j}}, {w_{1}^{j}}, \ldots , {w_{k}^{j}}$ be decreasing sequences
 in $\mathcal{QSH}_{m}(\Omega)$ converging respectively to $u, v, w_{1}, \ldots , w_{k}$ point-wisely as $j \longrightarrow\infty$.
Then the currents
$\gamma (u_{j}, v_{j}) \wedge \Delta w_{1}^{j} \wedge \ldots \wedge \Delta w_{k}^{j}\wedge \beta^{n-m} \longrightarrow \gamma (u, v) \wedge \Delta w_{1} \wedge \ldots \wedge \Delta w_{k}\wedge\beta^{n-m} $
 weakly as $j \longrightarrow\infty$.
\item[3)] The mixed product $\gamma (u, u) \wedge \Delta w_{1} \wedge \ldots\wedge \Delta w_{k}\wedge\beta^{n-m}$ is well defined as a $m$-positive $(2n-2m+2k+2)$-current.
\item[4)]  The Chern–Levine–Nirenberg
type estimate also holds for the $m$-positive current $\gamma (u, u) \wedge \Delta w_{1} \wedge \ldots\wedge \Delta w_{k}\wedge\beta^{n-m}.$
\end{enumerate}
\end{pro}
\begin{proof}
(1) From the polarization identity $$2(d_0 u\wedge d_1 v + d_0 v \wedge d_1 u)=\Delta [(u+ v)^2] - \Delta (u^2) -\Delta (v^2)-2u\Delta v -2v\Delta u.$$
It follows that \begin{equation}\label{11}
4\gamma (u, v) \wedge T=\Delta [(u+ v)^2]\wedge T - \Delta (u^2)\wedge T -\Delta (v^2)\wedge T-2u\Delta v\wedge T -2v\Delta u\wedge T,  
\end{equation}
where $T:= \Delta w_{1} \wedge \ldots\wedge \Delta w_{k}\wedge\beta^{n-m}.$ By \cite[(14)]{Ba} each term of the right hand side of (\ref{11}) is defined inductively as current. Thus the left side of (\ref{11}) is well defined.

(2) It follows from the first statement and the lemma \ref{lem2}. Since  $u$ and $v$ are bounded, we can let $u, \;v \geq 0$ by adding a positive constant. So $u^2$ and $v^2$ are also in $\mathcal{QSH}_{m}(\Omega).$ Therefore Lemma \ref{lem2} can be applied to $u^2,\;v^2$ and $(u+v)^2.$ 

To prove (3), it suffice to prove the positvity of the current $\gamma (u, u) \wedge \Delta w_{1} \wedge \ldots\wedge \Delta w_{k}\wedge\beta^{n-m}.$ Let $(u_{j})$ be a decreasing sequence
 in $\mathcal{QSH}_{m}(\Omega)$ converging  to $u$ as $j\rightarrow \infty.$  it follows from \cite[Lemma 3.1]{Wa1} that $\gamma(u_j,u_j)$ is a positive $2$-form, thus is strongly positive. For any strongly positive test form $\psi$ we have 
 $$\begin{array}{ll}
      [\gamma (u, u) \wedge \Delta w_{1} \wedge \ldots\wedge \Delta w_{m}\wedge\beta^{n-m}](\psi)&= \displaystyle\lim_{j \rightarrow \infty}[\gamma (u_j, u_j) \wedge \Delta w_{1} \wedge \ldots\wedge \Delta w_{k}\wedge\beta^{n-m}](\psi)   \\
      &  =\displaystyle \lim_{j \rightarrow \infty}[ \Delta w_{1} \wedge \ldots\wedge \Delta w_{k}\wedge\beta^{n-m}](\gamma (u_j, u_j) \wedge \psi) \geq 0\\
 \end{array}$$
 The last inequality follows from that the form $\gamma (u_j, u_j) \wedge \psi $ is  strongly positive (c.f proposition 3.1 in \cite{WW2}). 
 
 (4) The Chern-Leving-Nirenberg type estimate follows from \cite[Lemma 2.8]{Ba} and (\ref{11}). 
 
\end{proof}
\begin{lem}\label{lem20}
 Let $u, v, w_{1}, \ldots , w_{m-1} \in \mathcal{QSH}_{m}(\Omega) \cap L_{Loc}^{\infty}(\Omega)$. Then
 $$  \int_{\Omega}\gamma(u,v)\wedge T \leq \Big(\int_{\Omega}\gamma(u,u)\wedge T\Big)^{\dfrac{1}{2}}.\Big(\int_{\Omega}\gamma(v,v)\wedge T \Big)^{\dfrac{1}{2}}.$$
Where $T = \Delta w_{1} \wedge \ldots \wedge \Delta w_{m-1}\wedge \beta^{n-m}$.
\end{lem}

\begin{proof}
This follows from the above statements and \cite[Lemma 3.1]{Wa1}.
\end{proof}
\begin{pro}\label{pro11}
Suppose that $u, v \in \mathcal{QSH}_{m}^{-}(\Omega)\cap  L_{Loc}^{\infty}(\Omega).$
 If $\displaystyle{\lim_{q\longrightarrow \partial\Omega}}u(q)=0$, then
$$  \int_{\Omega} v \Delta u \wedge T \leq \int_{\Omega}u \Delta v \wedge T$$
where  $T = \Delta w_{1} \wedge \ldots \wedge \Delta w_{m-1}\wedge \beta^{n-m}.$

 Moreover, if $\displaystyle{\lim_{q \longrightarrow \partial \Omega}}v(q)=0$, then
$$  \int_{\Omega}v \Delta u \wedge T = \int_{\Omega}u \Delta v \wedge T= \int_{\Omega}- \gamma(u,v)\wedge T$$
\end{pro}
\begin{proof}
First,
let $\psi \in C_{0}^{\infty}(\Omega)$, by the induction definition \cite[(14)]{Ba} we obtain $$\int_{\Omega}\psi \Delta v \wedge T=\int_{\Omega} v \Delta \psi \wedge T  .$$ Hence $$\int_{\Omega}\psi \Delta v \wedge T\leq \int_{\Omega} v \Delta \psi \wedge T   \mbox{ for } \psi \in C_{0}^{\infty}(\Omega).$$

Let $u\in \mathcal{QSH}_{m}^{-}(\Omega)\cap  L_{Loc}^{\infty}(\Omega)$, and denote $u_\epsilon=\max\{u,-\epsilon\}.$ Then $u-u_\epsilon=\min\{0,u+\epsilon\}$ is a compactly supported function decreasing uniformly to $u$ as $\epsilon\longrightarrow 0$, thus $$ \int_{\Omega}u \Delta v \wedge T=\displaystyle \lim_{\epsilon\rightarrow 0}\int_{\Omega}(u-u_\epsilon) \Delta v \wedge T.$$ Using the above statement, we conclude that $$ \int_{\Omega}(u-u_\epsilon)_\delta \Delta v \wedge T=\displaystyle \lim_{\delta\rightarrow 0}\int_{\Omega}(u-u_\epsilon)_\delta \Delta v \wedge T=\displaystyle \lim_{\delta\rightarrow 0}\int_{\Omega} v\Delta  (u-u_\epsilon)_\delta \wedge T,$$ where $(u-u_\epsilon)_\delta$ is the standard regularization of $u-u_\epsilon$ such that $(u-u_\epsilon)_\delta\searrow u-u_\epsilon$ as $\delta\longrightarrow0.$ Fix an open set $\Omega^{'}\Subset \Omega$ such that the set $\{u<-\epsilon\} \Subset \Omega^{'}$, then $Supp\{\Delta(u-u_\epsilon)_\delta\} \subset \Omega^{'}$ for $\delta$ small enough. Note that $(u_\epsilon)_\delta \in \mathcal{QSH}_{m}(\Omega)$, then by  Lemma (\ref{lem1}) we get $\Delta(u_\epsilon)_\delta \wedge T \geq 0$. It follows from Lemma \ref{lem2} that
$$\begin{array}{ll}
\displaystyle \int_{\Omega}(u-u_\epsilon)_\delta \Delta v \wedge T& = \displaystyle \lim_{\delta\rightarrow 0}\int_{\Omega^{'}} v\Delta  (u-u_\epsilon)_\delta \wedge T\\
& \geq \displaystyle \limsup_{\delta \longrightarrow 0} \int_{\Omega^{'}} v\Delta  u_\delta \wedge T \\
&\displaystyle=\int_{\Omega^{'}} v\Delta  u \wedge T   .\\
\end{array}$$
 For an arbitrary $\Omega$,  letting $\epsilon\longrightarrow 0$ we obtain $\displaystyle\int_{\Omega} v \Delta u \wedge T \leq \int_{\Omega}u \Delta v \wedge T.$
 To show the second equality, it suffices to prove the second identity for the smooth case and repeat the above argument for the general case.
 Since $T$ is closed, applying Lemmas (\ref{lem0}) and  (\ref{lem1}) we get
 $$\begin{array}{ll}
\displaystyle\int_{\Omega} v \Delta u \wedge T &\displaystyle = \int_{\Omega}u \Delta v\wedge  T\\  &=\displaystyle\dfrac{1}{2} \int_{\Omega} u(d_0 d_1- d_1 d_0)v \wedge T \\
 & \displaystyle = \frac{1}{2} \int_{\Omega} ud_0(d_1v \wedge T)-\frac{1}{2} \int_{\Omega} ud_1(d_0v \wedge T) \\
& \displaystyle = - \frac{1}{2} \int_{\Omega} d_0 u \wedge d_1 v \wedge T)+\frac{1}{2} \int_{\Omega} d_1 u \wedge d_0 v \wedge T)\\
& \displaystyle = - \int_{\Omega} \gamma(u,v)\wedge T.\\
\end{array}$$
\end{proof}
\begin{pro}\label{pro12}
Suppose that $h, u_{1}, u_{2}, v_{1}, \ldots , v_{m-p-q} \in \mathcal{E}_{m}^{0}(\Omega),\;\; 1 \leq p, q < m.$ Let $T = \Delta v_{1} \wedge \ldots \wedge \Delta v_{m-p-q}\wedge \beta^{n-m}$.
 Then,$$\int_{\Omega}-h(\Delta u_{1})^{p}\wedge (\Delta u_{2})^{q}\wedge T \leq \Big[\int_{\Omega}-h(\Delta u_{1})^{p+q}\wedge T\Big]^{\frac{p}{p+q}}  \Big[\int_{\Omega}-h(\Delta u_{2})^{p+q}\wedge T\Big]^{\frac{q}{p+q}}$$
\end{pro}
\begin{proof}
For the case $p=q=1$, by Proposition \ref{pro11} and lemma \ref{lem20} we have
$$\begin{array}{ll}
\displaystyle \int_{\Omega} -h \Delta u_{1}\wedge \Delta u_{2} \wedge T &\displaystyle = \int_{\Omega}-u_{1} \Delta u_{2}\wedge \Delta h \wedge T= \int_{\Omega} \gamma(u_{1},u_{2})\wedge T\\
&\displaystyle{\leq \Big[\int_{\Omega} \gamma(u_{1},u_{1}) \wedge \Delta h \wedge T\Big]^{\frac{1}{2}}  \Big[\int_{\Omega} \gamma(u_{2},u_{2})\wedge \Delta h \wedge T\Big]^{\frac{1}{2}}}\\
&\displaystyle{= \Big[\int_{\Omega}-u_{1} \Delta u_{1}\wedge \Delta h \wedge T \Big]^{\frac{1}{2}}  \Big[\int_{\Omega}-u_{2} \Delta u_{2}\wedge \Delta h \wedge T \Big]^{\frac{1}{2}}}\\
&\displaystyle{= \Big[\int-h(\Delta u_{1})^{2}\wedge T\Big]^{\frac{1}{2}}  \Big[\int-h(\Delta u_{2})^{2}\wedge T\Big]^{\frac{1}{2}}}.\\

\end{array}$$
By induction, following the lines in \cite[Lemma 5.4]{C4} we get the desired result.
\end{proof}
\begin{cor}\label{3}
 Suppose that $ h, u_{1},\ldots, u_{m} \in \mathcal{E}_{m}^{0}(\Omega)$. Then,
$$\int-h\Delta u_{1}\wedge \ldots \Delta u_{m}\wedge \beta^{n-m} \leq \Big[\int-h(\Delta u_{1})^{m}\wedge \beta^{n-m}\Big]^{\frac{1}{n}} \ldots \Big[\int-h(\Delta u_{m})^{m}\wedge \beta^{n-m}\Big]^{\frac{1}{n}}$$
\end{cor}

\begin{thm}\label{th4}
Suppose that $u^{p} \in \mathcal{E}_{m}(\Omega),\; p = 1,\ldots, m$ and $(g_{j}^{p})_{j}\subset \mathcal{E}_{m}^{0}(\Omega)$  such that $g_{j}^{p}\downarrow   u_{p},\; \forall p.$ Then, the sequence of measures $\Delta g_{j}^{1}\wedge \ldots\wedge \Delta g_{j}^{m}\wedge \beta^{n-m}$ converges weakly to a
 positive Radon measure which does not depend on the choice of the sequences $(g_{j}^{p})_{j}.$
 We then define $\Delta g_{j}^{1}\wedge \ldots\wedge \Delta g_{j}^{m}\wedge \beta^{n-m}$ to be this weak limit.
\end{thm}
\begin{proof}
By lemma \ref{lem21} and proceeding as in  the proof of \cite[Theorem 4.2]{C4} we get the result.
\end{proof}
\begin{pro}\label{pro13}
Suppose that $u^{p} \in \mathcal{F}_{m}(\Omega),\; p = 1,\ldots, m$ and $(g_{j}^{p})_{j}\subset \mathcal{E}_{m}^{0}\cap C(\Omega)$  such that $g_{j}^{p}\downarrow   u_{p},\; \forall p.$ and
$$   \sup_{j,p}\int_{\Omega}(\Delta g_{j}^{p})^{m}\wedge \beta^{n-m}< +\infty.$$
 Then, for each $h \in \mathcal{E}_{m}^{0}(\Omega)\cap C(\Omega)$ we have $$\lim_{j\longrightarrow +\infty}\int_{\Omega}h\Delta g_{j}^{1}\wedge \ldots\wedge \Delta g_{j}^{m}\wedge \beta^{n-m}=\int_{\Omega}h\Delta u^{1}\wedge \ldots\wedge \Delta u^{m}\wedge \beta^{n-m}.$$
\end{pro}
\begin{proof}
Clearly we have
\begin{equation}\label{6}
\sup_{j} \int_{\Omega}\Delta g_{j}^{1}\wedge \ldots\wedge \Delta g_{j}^{m}\wedge \beta^{n-m}<+\infty.
\end{equation}
Let  $h \in \mathcal{E}_{m}^{0}(\Omega)\cap C(\Omega)$ and for $\epsilon>0$ small enough, we consider the function $h_{\epsilon}=\max\lbrace h,-\epsilon\rbrace$. Then $h-h_{\epsilon}$ is continuous and compactly supported in $\Omega$. Applying Theorem \ref{th4} we get

$$\lim_{j\longrightarrow +\infty}\int_{\Omega}(h-h_{\epsilon})\Delta g_{j}^{1}\wedge \ldots\wedge \Delta g_{j}^{m}\wedge \beta^{n-m}=\int_{\Omega}(h-h_{\epsilon})\Delta u^{1}\wedge \ldots\wedge \Delta u^{m}\wedge \beta^{n-m}.$$
From (\ref{6}) and the fact that $\vert h_{\epsilon}\vert<\epsilon$ follow the result.
\end{proof}
\begin{cor}\label{cor3}
Suppose that $(g_{j})_{j}\subset \mathcal{E}_{m}^{0}(\Omega) $ decreases to $u \in \mathcal{F}_{m}(\Omega),\; j \longrightarrow +\infty$, such that
$$ \sup_{j}\int_{\Omega}(\Delta g_{j})^{m}\wedge \beta^{n-m}< +\infty.$$
Then, for each $ h \in \mathcal{E}_{m}^{0}(\Omega)$, the sequence of measures $h(\Delta g_{j})^{m}\wedge \beta^{n-m}$ converges weakly to $h(\Delta u)^{m}\wedge \beta^{n-m}$.
\end{cor}
\begin{cor}\label{cor4}
 Suppose that $  u_{1},\ldots, u_{m} \in \mathcal{F}_{m}(\Omega)$. Then,
$$\int\Delta u_{1}\wedge \ldots \Delta u_{m}\wedge \beta^{n-m} \leq \Big[\int(\Delta u_{1})^{m}\wedge \beta^{n-m}\Big]^{\frac{1}{n}} \ldots \Big[\int(\Delta u_{m})^{m}\wedge \beta^{n-m}\Big]^{\frac{1}{n}}$$
\end{cor}
\begin{proof}
It follows from  Definition \ref{def1} and Proposition \ref{pro13} that Corollary \ref{cor4} holds.
\end{proof}
\begin{thm}\label{th5}
Let $u,v,w_{1},\ldots,w_{m-1}\in \mathcal{F}_{m}(\Omega)$ and $T=\Delta w_{1}\wedge \ldots \wedge \Delta w_{m-1}\wedge \beta^{n-m}.$ Then $$\int_{\Omega}u \Delta v\wedge T=\int_{\Omega}v \Delta u\wedge T. $$
\end{thm}
\begin{proof}
 Let $u_j,v_j,w_{1}^j,\ldots,w_{m-1}^j$ be sequences in $ \mathcal{E}_{m}^{0}(\Omega)\cap C(\Omega)$ decreasing to $u,v,w_{1},\ldots,w_{m-1}$ respectively

such that their total masses are uniformly bounded
$$ \sup_{j} \int_{\Omega} \Delta v_{j} \wedge T_{j} \wedge \beta^{n-m} < + \infty\;,\; \sup_{j} \int_{\Omega} \Delta u_{j} \wedge T_{j} \wedge \beta^{n-m}< + \infty,$$ where $T_{j}= \Delta w_{1} \wedge \ldots \Delta w_{m-1} \wedge \beta^{n-m}.$ By Theorem \ref{th4}  we obtain $\Delta u_{j} \wedge T_{j} \wedge \beta^{n-m}$ converges weakly to $\Delta u \wedge T \wedge \beta^{n-m}.$ For each fixed $K \in \mathbb{N}$ and any $j > k$ we have $$ \int_{\Omega}v_{k} \Delta u_{k} \wedge T_{k} \wedge \beta^{n-m} \geq \int_{\Omega}v_{k} \Delta u_{j} \wedge T_{j} \wedge \beta^{n-m} \geq \int_{\Omega}v_{j} \Delta u_{j} \wedge T_{j} \wedge \beta^{n-m}. $$
We then infer that the sequence of real numbers $ \displaystyle {\int_{\Omega}v_{j} \Delta u_{j} \wedge T \wedge \beta^{n-m}}$ decreases to some $a \in \mathbb{R} \cup \{+\infty \}.$ Using proposition \ref{pro13} and letting $j \longrightarrow +\infty$ we get $$ \int_{\Omega}v_{k} \Delta u \wedge T \wedge \beta^{n-m},$$
from which we obtain $\displaystyle \int_{\Omega}v \Delta u \wedge T \wedge \beta^{n-m} \geq a. $ For each fixed $k$  we have $$
\begin{array}{ll}
     \displaystyle{\int_{\Omega}v \Delta u \wedge T \wedge \beta^{n-m}}
       & \displaystyle \leq \int_{ \Omega}v_{k} \Delta u \wedge T \wedge \beta^{n-m}\\ & = \displaystyle\lim_{j \longrightarrow +\infty}\int_{\Omega}v_{k} \Delta u_{j} \wedge T_{j} \wedge \beta^{n-m} \\
     & \displaystyle \leq \int_{\Omega}v_{k} \Delta u_{k} \wedge T_{k} \wedge \beta^{n-m}
\end{array}
$$
This  implies that $\displaystyle{\int_{\Omega}v \Delta u \wedge T \wedge \beta^{n-m} =a,}$ from which the result follows.
\end{proof}
\begin{defi}
We define the quaternionic $p$-energy $(p>0)$ of $\varphi\in \mathcal{E}_{m}^{0}(\Omega)$ to be $$E_{p}(\varphi):=\int_{\Omega}(-\varphi)^{p}(\Delta \varphi)^{m}\wedge \beta^{n-m},$$
if $p = 1$ we  denote by $E(\varphi) = E_{1}(\varphi).$
and the mutual quaternionic $p$-energy of $\varphi_{0}\ldots,\varphi_{m}\in \mathcal{E}_{m}^{0}(\Omega)$ to be
$$ E_{p}(\varphi_{0},\varphi_{1},\ldots,\varphi_{m}):=\int_{\Omega}(-\varphi_{0})^{p}\Delta \varphi_{1}\wedge \ldots \wedge \Delta \varphi_{m}\wedge \beta^{n-m},\;\;\;p\geq 1.$$
\end{defi}

The following  H\"{o}lder-type inequality  plays an important role in the variational approach in the next section.
\begin{thm}\label{th6}
Let $u,v_{1},\ldots,v_{m} \in \mathcal{E}_{m}^{0}(\Omega)$ and $p\geq 1$. We have $$E_{p}(u,v_{1},\ldots,v_{m})\leq D_{p}E_{p}(u)^{\frac{p}{m+p}}E_{p}(v_{1})^{\frac{1}{m+p}}\ldots E_{p}(v_{m})^{\frac{1}{m+p}}$$ where $D_{1}=1, \;\; D_{p}=p^{p\alpha(m,p)/p-1}$ for $p>1$ and $$\alpha(m,p)=(p+2)\Big(\dfrac{p+1}{p}\Big)^{m-1}-(p+1).$$
\end{thm}
\begin{proof}
 Let $F(u,v,v_{1},\ldots ,v_{m-1}):= \displaystyle\int_{\Omega}(-u)^{p}\Delta v_{1}\wedge \ldots \wedge \Delta v_{m-1}\wedge \beta^{n-m},$ for $p \geq 1$ and $ u,v,v_{1},\ldots,v_{m-1} \in \mathcal{E}_{m}^{0}(\Omega)$. By using \cite[Theorem 4.1]{Pe} it suffices to prove that
$$F(u,v,v_{1},\ldots,v_{m-1})\leq C(p)F(u,u,v_{1},\ldots,v_{m-1})^{\frac{p}{p+1}}F(v,v,v_{1},\ldots,v_{m-1})^{\frac{1}{p+1}}$$
where $C(p) = 1$ if $p = 1$ and $C(p) = p^{\frac{p}{p-1}}$ if $p>1$.
 Set $T=\Delta v_{1}\wedge \ldots \wedge \Delta v_{m-1}\wedge \beta^{n-m}.$\\ If $p = 1$, the above inequality
becomes
$$\int_{\Omega}(-u)(\Delta v\wedge T \leq \Big(\int_{\Omega}(-u)(\Delta u)\wedge T\Big)^{\frac{1}{2}}  \Big(\int_{\Omega}(-v)(\Delta v \wedge T\Big)^{\frac{1}{2}}$$
which is the Cauchy–Schwarz inequality.\\ If $p > 1,$ by Proposition \ref{pro11} we get 
\begin{equation}\label{eq26} \begin{array}{ll}
\displaystyle\int_{\Omega}(-u)^{p}\Delta v\wedge T & = \displaystyle\int_{\Omega}-\gamma((-u)^{p},v)\wedge T\\
&=\displaystyle p\int_{\Omega}(-u)^{p-1}\gamma(u,v)\Delta v\wedge T \\
&=\displaystyle p\int_{\Omega}(-v)\gamma((-u)^{p-1},u)\Delta v\wedge T + p\int_{\Omega}(-v)(-u)^{p-1}\Delta u\wedge T\\ 
& = \displaystyle-p(p-1)\int_{\Omega}(-u)^{p-2}\gamma(u,u)\Delta v\wedge T + p\int_{\Omega}(-v)(-u)^{p-1}\Delta u\wedge T\\ 
&\displaystyle 
\leq p \displaystyle\int_{\Omega}(-v)(-u)^{p-1}\Delta u\wedge T.
\end{array} \end{equation} 
The last inequality follows from the fact that $\gamma(u,u)\Delta v\wedge T$  is a $m$-positive current by Proposition \ref{pro111}.
Using H\"{o}lder’s inequality we get
$$\int_{\Omega}(-u)^{p}\Delta v\wedge T \leq p \Big( \int_{\Omega}(-u)^{p}\Delta u\wedge T\Big)^{\frac{p-1}{p}}\Big( \int_{\Omega}(-v)^{p}\Delta u\wedge T\Big)^{\frac{1}{p}}.$$
Replacing $u$ by $v$ in the above inequality we obtain
$$\int_{\Omega}(-v)^{p}\Delta u\wedge T \leq p \Big( \int_{\Omega}(-v)^{p}\Delta v\wedge T\Big)^{\frac{p-1}{p}}\Big( \int_{\Omega}(-u)^{p}\Delta v \wedge T\Big)^{\frac{1}{p}}.$$
the result is a consequence of the two last inequalities.
\end{proof}
\begin{defi}\label{def2}
Denote by $\mathcal{K}\subset \mathcal{QSH}_{m}^{-}(\Omega)$ such that:
\begin{enumerate}
\item[1)] If $u \in \mathcal{K},\; v\in \mathcal{QSH}_{m}^{-}(\Omega)$, then $\max\lbrace u,v\rbrace \in \mathcal{K}.$
\item[2)]If $u \in \mathcal{K},\; \varphi_{j}\in \mathcal{QSH}_{m}^{-1}(\Omega)\cap L_{loc}^{\infty}(\Omega),\;\varphi_{j}\searrow u,\;j\longrightarrow +\infty,$ then $(\Delta\varphi_{j})^{m}\wedge \beta^{n-m}$ is weakly convergent.
\end{enumerate}
\end{defi}
\begin{cor}
By $\mathcal{E}$ we denote one of the classes  $\mathcal{E}_{m}^{0}(\Omega), \mathcal{E}_{m}(\Omega), \mathcal{F}_{m}(\Omega), \mathcal{E}_{m}^{p}(\Omega), \mathcal{F}_{m}^{p}(\Omega),\; p>  0$. We have the following properties
\begin{enumerate}
\item[1)]$\mathcal{E}$ is convex and have the property (1) in Definition \ref{def2}.
\item[2)] if $\mathcal{E}=\mathcal{E}_{m}(\Omega)$, then $\mathcal{E}$ has properties (1) and (2)
in Definition \ref{def2}.
\item[3)] $\mathcal{E}_{m}(\Omega)$ is the largest class for which the properties of  Definition \ref{def2}  hold true.
\end{enumerate}
\end{cor}
\begin{cor}\label{cor5}
Let $u,v \in \mathcal{E}_{m}^{0}(\Omega)$ and  $u \leq v$. Then, $E_{p}(v) \leq AE_{p}(u)$, where the constant $A$ is independent of $u,v.$ In particular, for $p=1$ we have $E_{1}(v) \leq E_{1}(u)$.
\end{cor}
\begin{proof}
Applying  Theorem \ref{th6} directly we get the desired result.
\end{proof}
\begin{cor}\label{cor6}
Let $V$ be an open subset of $\Omega$ and $\varphi\in\mathcal{E}_{m}^{0}(\Omega),\;p\geq1.$ Then
$$\int_{V}(\Delta\varphi)^{m}\wedge \beta^{n-m}\leq MC_{m}(V)^{\frac{p}{p+m}}E_{p}(\varphi)^{\frac{m}{p+m}},$$ where $M$ is a constant depending only on $p$ and $m$ and $C_{m}(V)$ is the quaternionic $m$-capacity of $V$.
\end{cor}
\begin{proof}
We can suppose that $V$ is relatively compact in $\Omega$. Denote by $u = u_{m,V,\Omega}^{*}$ the regularized $m$-extremal function
of $V$ in $\Omega$. Then $u \in \mathcal{E}_{m}^{0}(\Omega)$
 and $u = -1$ in $V$. From Theorem \ref{th6} we have
 $$\begin{array}{ll}
 \displaystyle{\int}_{V}(\Delta\varphi)^{m}\wedge \beta^{n-m}&\leq \displaystyle{\int}_{\Omega}(-u)^{p}(\Delta\varphi)^{m}\wedge \beta^{n-m}\\ &\displaystyle\leq D_{p}E_{p}(u)^{\frac{p}{p+m}}E_{p}(\varphi)^{\frac{m}{p+m}}\\
 &\leq D_{p}\Big(\displaystyle{\int}_{\Omega}(\Delta\varphi)^{m}\wedge \beta^{n-m}\Big)^{\frac{p}{p+m}}E_{p}(\varphi)^{\frac{m}{p+m}} \\ &\displaystyle
\leq D_{p}C_{m}(V)^{\frac{p}{p+m}}E_{p}(\varphi)^{\frac{m}{p+m}}.
 \end{array}$$
\end{proof}
\begin{thm}\label{th7}
If $u^{k} \in \mathcal{E}_{m}^{p}(\Omega),\;k=1,\ldots,m,\;p \geq 1$ and $(g_{j}^{k})_{j} \subset \mathcal{E}_{m}^{0}(\Omega)$ decreases to $u^{k},\;j \longrightarrow + \infty$ such that $$\displaystyle{\sup_{j,k}} E_{p}(g_{j}^{k}) < +\infty.$$ Then, the sequence of measures $\Delta g_{j}^{1} \wedge \ldots \wedge \Delta g_{j}^{m}\wedge \beta^{n-m}$ is weakly convergent to a positive measure and the limit does not depend to the particular sequence. We then define $\Delta u^{1}\wedge \ldots\wedge \Delta u^{m}\wedge\beta^{n-m}$ to be this weak limit.
\end{thm}
\begin{proof}
 Let $K$ be a compact
subset of $\Omega$. For each $j \in \mathbb{N},\; k = 1,\ldots, m$ consider
$$
h_{j}^{k} := \sup \lbrace u\in \mathcal{QSH}_{m}^{-}(\Omega) \;/\; u \leq g_{j}^{k} \mbox{ on }  K \rbrace.$$
Then by using a standard balayage argument we see that $supp((\Delta h_{j}^{k})^{m}\wedge\beta^{n-m})\subset K$. It follows that $ h_{j}^{k}$
decreases to $v^{k}\in \mathcal{F}_{m}^{p}(\Omega)$. We have also $v^{k} = u^{k}$ on $K.$
Now, fix $h\in \mathcal{E}_{m}^{0}(\Omega)$, then
$$\int_{\Omega}h\Delta g_{j}^{1}\wedge \ldots \wedge \Delta g_{j}^{m} \wedge \beta^{n-m}$$
is decreasing to a finite number. Thus  $\displaystyle{\lim_{j}\int}_{\Omega}h \Delta g_{j}^{1}\wedge \ldots\wedge \Delta g_{j}^{m} \wedge \beta^{n-m} $ exists for every
$h \in \mathcal{E}_{m}^{0}(\Omega).$
 From  Proposition \ref{pro13} follows the weak convergence of the sequence
$\Delta g_{j}^{1}\wedge \ldots \wedge \Delta g_{j}^{m} \wedge \beta^{n-m}.$
To prove the last statement, it is sufficient to repeat the arguments in the proof of \cite[Theorem 4.2]{C4}.
\end{proof}
\begin{cor}\label{cor66}
Let $u^{k} \in \mathcal{E}_{m}^{p}(\Omega),\;k=1,\ldots,m,\;p \geq 1$ and $(u_{j}^{k})_{j} \subset \mathcal{E}_{m}^{0}(\Omega)$ decreases to $u^{k},\;j \longrightarrow + \infty$ such that $$\displaystyle{\sup_{j,k}} E_{p}(u_{j}^{k}) < +\infty.$$ Then,$$  \displaystyle \lim_{j \longrightarrow +\infty} \int_{\Omega}h \Delta u_{1}^{j} \wedge \ldots \wedge u_{m}^{j} \wedge \beta^{n-m}=\int_{\Omega}h \Delta u_{1} \wedge \ldots \wedge u_{m} \wedge \beta^{n-m} ,$$ for each $h \in \mathcal{E}_{m}^{0}(\Omega) \cap C(\Omega). $
\end{cor}
\begin{proof}
We repeat the same arguments in the proof of Proposition \ref{pro13}.
\end{proof}
\begin{cor}\label{cor7}
If $u^{k} \in \mathcal{E}_{m}^{p}(\Omega),\;k=1,\ldots,m,\;p \geq 1$ and $(g_{j}^{k})_{j} \subset \mathcal{E}_{m}^{p}(\Omega)$ decreases to $u^{k},\;j \longrightarrow + \infty$, Then, the sequence of measures $\Delta g_{j}^{1} \wedge \ldots \wedge \Delta g_{j}^{m}\wedge \beta^{n-m}$ is weakly convergent to $\Delta u^{1}\wedge \ldots\wedge \Delta u^{m}\wedge\beta^{n-m}.$
\end{cor}
\begin{proof}
Similarly as in the proof of Theorem \ref{th7}.
\end{proof}
\begin{thm}\label{th8}
Let $u,v,w_{1},\ldots,w_{m-1}\in \mathcal{E}_{m}^{p}(\Omega)$ and $T=\Delta w_{1}\wedge \ldots \wedge \Delta w_{m-1}\wedge \beta^{n-m}.$ Then $$\int_{\Omega}u \Delta v\wedge T=\int_{\Omega}v \Delta u\wedge T. $$
\end{thm}
\begin{proof}
Thanks to Theorem \ref{th7}  the same arguments as in the proof of Proposition \ref{pro13}  can be used here.
\end{proof}
\begin{pro}
Let $u,v\in \mathcal{E}_{m}^{p}(\Omega)$ (or $\mathcal{F}(\Omega)$) such that $u\leq v$ on $\Omega$. Then $$\int_{\Omega} (\Delta u)^{m}\wedge \beta^{n-m} \geq \int_{\Omega} (\Delta v)^{m}\wedge \beta^{n-m}. $$
\end{pro}
\begin{proof}
Let  $(u_{j}),(v_{j})$  be two sequences in $  \mathcal{E}_{m}^{0}(\Omega)$ decreasing to $u, v$ as in the definition of $\mathcal{E}_{m}^{p}(\Omega)$ . Fix
$h \in \mathcal{E}_{m}^{0}(\Omega)\cap C(\Omega).$
 We can suppose that $u_{j} \leq v_{j},\;\; \forall j$ in $\Omega$. Then integrating by parts we get
$$\int_{\Omega}(-h) (\Delta u_{j})^{m}\wedge \beta^{n-m} \geq \int_{\Omega}(-h) (\Delta v_{j})^{m}\wedge \beta^{n-m}. $$
From Theorem \ref{th7}, Proposition \ref{pro13} and letting $j\longrightarrow +\infty$ follow
$$\int_{\Omega}(-h) (\Delta u)^{m}\wedge \beta^{n-m} \geq \int_{\Omega}(-h) (\Delta v)^{m}\wedge \beta^{n-m}. $$

to get the result it suffices to let  h decreases to $-1$.
\end{proof}
\begin{pro}\label{pro14}
\begin{enumerate}
\item[1)] If $u \in \mathcal{E}_{m}^{1}(\Omega)$, then $\displaystyle{\int}_{\Omega}(-u)(\Delta u)^{m}\wedge \beta^{n-m}< +\infty$.
\item[2)] If $(u_{j})$ is a sequence in $\mathcal{E}_{m}^{0}(\Omega)$ decreasing to $u$, then $$\displaystyle{\int}_{\Omega}(-u_{j})(\Delta u_{j})^{m}\wedge \beta^{n-m} \nearrow \displaystyle{\int}_{\Omega}(-u)(\Delta u)^{m}\wedge \beta^{n-m}.$$
\end{enumerate}
\end{pro}
\begin{proof}
(1) Let $(u_{j})$ be a sequence in $\mathcal{E}_{m}^{0}(\Omega)$ such that $u_{j}\searrow u$ and $ \displaystyle \sup_{j}\int_{\Omega}(-u_{j})(\Delta u_{j})^{m}\wedge \beta^{n-m} < +\infty $, then
$$\displaystyle \int_{\Omega}(-u)(\Delta u)^{m}\wedge \beta^{n-m} \leq  \displaystyle \liminf _{j\longrightarrow + \infty}\int_{\Omega}(-u)(\Delta u)^{m}\wedge \beta^{n-m}< +\infty .$$
(2) From Theorem \ref{th7} follows $(\Delta u_{j})^{m} \wedge \beta^{n-m} \rightharpoonup (\Delta u )^{m} \wedge \beta^{n-m}.$ Since $(-u_{j}) \nearrow (-u)$ and are lower semi-continuous, then we obtain $$\displaystyle \int_{\Omega}(-u)(\Delta u)^{m}\wedge \beta^{n-m} \leq  \displaystyle \liminf _{j \longrightarrow + \infty}\int_{\Omega}(-u_{j})(\Delta u_{j})^{m}\wedge \beta^{n-m}.$$
Then, it is sufficient to show that for each $j$ we have
$$\displaystyle \int_{\Omega}(-u)(\Delta u_{j})^{m}\wedge \beta^{n-m} \leq \displaystyle \int_{\Omega}(-u)(\Delta u)^{m}\wedge \beta^{n-m}  .$$ Let $h \in \mathcal{E}_{m}^{0}(\Omega)\cap C(\Omega)$ such that $u \leq h$. Since integration by part is allowed in $\mathcal{E}_{m}^{1}(\Omega)$, then the sequence $\displaystyle \Big(\int_{\Omega}(-h)(\Delta u_{j})^{m} \wedge \beta^{n-m}\Big)_{j}$ is increasing and by corollary \ref{cor66} its limit is $\displaystyle \int_{\Omega}(-h)(\Delta u)^{m}\wedge \beta^{n-m}.$
\end{proof}
\begin{lem}{(\cite[proposition 2.21]{Ba})}
 For $u,v \in \mathcal{QSH}_{m}(\Omega) \cap L_{loc}^{\infty}(\Omega)$, we have
 $$(\Delta \max \lbrace u,v \rbrace )^{m}\wedge \beta^{n-m}\geq \chi_{\lbrace u>v \rbrace}(\Delta u)^{m}\wedge \beta^{n-m}+ \chi_{\lbrace u\leq v \rbrace}(\Delta v)^{m}\wedge \beta^{n-m} $$ where $\chi_{A}$ is the characteristic function of a set $A.$
 \end{lem}
 \begin{thm}\label{th9}
 Let $u,u_{1},\ldots,u_{m-1} \in \mathcal{F}_{m}(\Omega)$ (or ($\mathcal{E}_{m}^{p}(\Omega)$),  $v \in \mathcal{QSH}_{m}^{-}(\Omega)$ and $T=\Delta u_{1}\wedge \ldots \wedge \Delta u_{m-1}\wedge \beta^{n-m}.$
 Then $$\Delta \max \lbrace u,v \rbrace \wedge T \vert_{\lbrace u>v \rbrace}=\Delta u\wedge T\vert_{\lbrace u> v \rbrace}. $$
 \end{thm}
 \begin{proof}
 First we prove the result in the case where $v\equiv b$ with some constant $b$. By Theorem \ref{th3} there exist $(u^{j})_{j}, (u_{k}^{j})_{j} \in \mathcal{E}_{m}^{0}(\Omega) \cap C(\overline{\Omega})$ such that $u^{j} \searrow u$ and $u_{k}^{j}\searrow u_{k}$ for each $k=1,\ldots,m-1$. Since $\{u^{j}>b \}$ is open, then $$\Delta \max \lbrace u^{j},b \rbrace \wedge T^{j} \vert_{\lbrace u^{j}>b \rbrace}=\Delta u^{j}\wedge T^{j}\vert_{\lbrace u^{j}> b \rbrace}. $$ where $T^{j}=\Delta u_{1}^{j}\wedge \ldots \wedge \Delta u_{m-1}^{j}\wedge \beta^{n-m}.$
 Because of $\{u>b \} \subset \{u^{j}> b\}$ we obtain that $$\Delta \max \lbrace u^{j},b \rbrace \wedge T^{j} \vert_{\lbrace u >b \rbrace}=\Delta u^{j}\wedge T^{j}\vert_{\lbrace u > b \rbrace}. $$
 Letting $j\longrightarrow +\infty,$ by Proposition \ref{pro13} and Theorem \ref{th7} follow
 $$\max(u-b,0)\Delta(\max(u^{j},b))\wedge T^{j}\longrightarrow \max(u-b,0)\Delta(\max(u,b))\wedge T$$
   $$\max(u-b,0)\Delta u^{j}\wedge T^{j}\longrightarrow \max(u-b,0)\Delta u \wedge T.$$
   Therefore,
    $$\max(u-b,0)\Delta  ( \max(u-b,0)) - \Delta u ] \wedge T=0.$$
   Which implies that $$\Delta  ( \max(u-b,0)) \wedge T = \Delta u  \wedge T \mbox{ on the set } \{u> b\}.$$
    For the general case, we repeat the same arguments in \cite[Theorem 4.1]{Van}.
\end{proof}
 \begin{cor}\label{cor8}
 For $p\geq 1$ and $u,v \in \mathcal{E}_{m}^{p}(\Omega)$ we have 
  $$\chi_{\lbrace u>v \rbrace}(\Delta \max \lbrace u,v \rbrace )^{m}\wedge \beta^{n-m} = \chi_{\lbrace u>v \rbrace}(\Delta u)^{m}\wedge \beta^{n-m}. $$ 
 \end{cor}
 
\section{The variational approach}\label{sec3}
\subsection{The energy functional}
\begin{defi}
\begin{enumerate}
\item[1)] For a positive Radon measure $\mu$ in $\Omega$, the energy functional $\mathcal{F}_{\mu}
:\mathcal{E}_{m}^{1}(\Omega)\longrightarrow \mathbb{R}$ is defined by
$$\mathcal{F}_{\mu}(u)=\dfrac{1}{m+1}E(u)+L_{\mu}(u),$$ where $L_{\mu}(u)=\displaystyle \int_{\Omega} ud \mu$ and $E(u)$ is the quaternionic energy of $u \in \mathcal{E}_{m}^{0}(\Omega).$
\item[2)]  We say $\mathcal{F}_{\mu}$ is proper if $\mathcal{F}_{\mu}\longrightarrow +\infty$ whenever $E\longrightarrow +\infty$.
\end{enumerate}
\end{defi}
\begin{defi}
We say that a positive measure $\mu$ belong to $\mathcal{M}_{p}$ if there exists a constant $A> 0$ such that
$$\int_{\Omega} (-u)^{p}d \mu\leq AE_{p}(u)^{\frac{p}{m+p}},\;\;\forall u \in \mathcal{E}_{m}^{p}(\Omega)$$
\end{defi}
The following remark will be proved as the complex case in \cite{C3}.
\begin{rem}
$\mu \in \mathcal{M}_{p}$ if and only if $\mathcal{E}_{m}^{p}(\Omega)\subset L^{p}(\Omega,\mu).$
\end{rem}
\begin{pro}\label{pro15}
\begin{enumerate}
\item[1)] If $(u_{j})\subset \mathcal{E}_{m}^{1}(\Omega)$ such that $\sup_{j}E(u_{j})< +\infty,$ then $(\sup_{j}u_{j})^{*}\in \mathcal{E}_{m}^{1}(\Omega).$
\item[2)] If $(u_{j})\subset \mathcal{E}_{m}^{1}(\Omega)$ such that $\sup_{j}E(u_{j})< +\infty,$ and $u_{j}\longrightarrow u$, then $u \in \mathcal{E}_{m}^{1}(\Omega).$
\item[3)] The functional $ E:  \mathcal{E}_{m}^{1}(\Omega)\longrightarrow \mathbb{R}$ is lower semi-continuous.
\item[4)] If $u, v \in \mathcal{E}_{m}^{1}(\Omega),$ then
$$ E(u+v)^{\frac{1}{m+1}} \leq E(u)^{\frac{1}{m+1}}+E(v)^{\frac{1}{m+1}}.$$
Moreover, if $\mu \in \mathcal{M}_{1}$, then $\mathcal{F}_{\mu}$ is proper and convex.
\item[5)] If $\mu \in \mathcal{M}_{1},\;u \in \mathcal{E}_{m}^{1}(\Omega)$ and $u_{j} \in \mathcal{E}_{m}^{0}(\Omega)$ such that $u_{j}\searrow u$, then $\displaystyle{\lim_{j\longrightarrow+\infty}}\mathcal{F}_{\mu}(u_{j})=\mathcal{F}_{\mu}(u).$
\item[6)]  If $u, v \in \mathcal{E}_{m}^{1}(\Omega),$ then
$$\int_{\lbrace u>v\rbrace}(\Delta u)^{m}\wedge \beta^{n-m}\leq \int_{\lbrace u>v\rbrace}(\Delta v)^{m}\wedge \beta^{n-m}.$$
\end{enumerate}
\end{pro}
\begin{proof}
\begin{enumerate}
\item[1)]  Let $(\varphi_{j})$ be a  sequence in $\mathcal{E}_{m}^{0}(\Omega) \cap C(\Omega)$ decreasing to  $u=(\sup_{j}u_{j})^{*}$. Since $u_{j} \leq \varphi_{j}$ and $\sup_{j}E(u_{j})< +\infty$ then  $\sup_{j}E(\varphi_{j})< +\infty.$ Hence $u \in \mathcal{E}_{m}^{1}(\Omega)$.
\item[2)]  Let $(\varphi_{j})$ be a  sequence in $\mathcal{E}_{m}^{0}(\Omega) \cap C(\Omega)$ decreasing to  $u$. Let denote  $ \psi_{j}:=\max \lbrace u_{j},\varphi_{j}\rbrace$. Then $\psi_{j} \in \mathcal{E}_{m}^{0}(\Omega)$ and $E(\psi_{j})\leq E(u_{j})$, which implies that $u \in \mathcal{E}_{m}^{1}(\Omega)$.
\item[3)]Suppose that $u, u_{j} \in \mathcal{E}_{m}^{1}(\Omega)$ such that $ u_{j}$ converges to $u \in L_{loc}^{1}(\Omega)$. For each $j \in \mathbb{N}$, the function
$\varphi_{j} := (\sup_{k\geq j}u_{k})^{*}$
is in $\mathcal{E}_{m}^{1}(\Omega)$ and $\varphi_{j} \downarrow u$. Hence $E(\varphi_{j}) \uparrow E(u)$. From $E(u_{j}) \geq E(\varphi_{j})$ follows $\lim \inf_{j}E(u_{j})\geq E(u)$.
\item[4)]It follows from Theorem \ref{th6} that

$$\begin{array}{ll}
E(u+v)&=\displaystyle{\int}_{\Omega}(-u)(\Delta(u+v)^{m}\wedge \beta^{n-m}+\displaystyle{\int}_{\Omega}(-v)(\Delta(u+v)^{m}\wedge \beta^{n-m}\\
&\leq E(u)^{\frac{1}{m+1}}E(u+v)^{\frac{m}{m+1}}+E(v)^{\frac{1}{m+1}}E(u+v)^{\frac{m}{m+1}}
\end{array}$$
which implies that $E^{\frac{1}{m+1}}$
 is convex since it is homogeneous of degree $1$. So, $E$ is also convex. If $\mu$ belongs to
$\mathcal{M}_{1}$,  there exists $ A > 0$ such that
 $\displaystyle{\int}_{\Omega} (-u)d \mu\leq AE(u)^{\frac{1}{m+1}}$, for every $u \in \mathcal{E}_{m}^{1}(\Omega)$.
Then we get
$\mathcal{F}_{\mu}(u)\geq \dfrac{1}{m+1}E(u)-AE(u)^{\frac{1}{m+1}}\longrightarrow +\infty .$
\item[5)] Let $u \in \mathcal{E}_{m}^{1}(\Omega)$ and $u_{j} \in \mathcal{E}_{m}^{0}(\Omega)$ such that $u_{j}\searrow u$,  then  by Proposition \ref{pro14} follows $E(u_{j}) \nearrow E(u)$. Applying  the monotone convergence
theorem and the fact that $\mu\in \mathcal{M}_{1}$ we get the result.
\item[6)] let $h \in \mathcal{E}_{m}^{0}(\Omega) \cap C(\Omega)$ and $u \in \mathcal{E}_{m}^{1}(\Omega)$. then  we have $$\displaystyle{\int}_{\Omega}(-h)(\Delta \max \lbrace u,v \rbrace)^{m}\wedge \beta^{n-m}\leq \displaystyle{\int}_{\Omega}(-h)(\Delta u)^{m}\wedge \beta^{n-m},$$
 since $$\displaystyle{\int}_{\Omega}(-h)(\Delta u)^{m}\wedge \beta^{n-m}\leq E(h)^{\frac{1}{m+1}}E(u)^{\frac{m}{m+1}}<+\infty $$
Then, it follows from Corollary \ref{cor8}  that
$$\begin{array}{ll}
\displaystyle{\int}_{\lbrace u> v\rbrace}(-h)(\Delta u)^{m}\wedge \beta^{n-m}&=\displaystyle{\int}_{\lbrace u> v \rbrace}(-h)(\Delta \max\lbrace  u,v \rbrace)^{m}\wedge \beta^{n-m}\\
&\leq \displaystyle{\int}_{\Omega}(-h)(\Delta \max \lbrace  u,v \rbrace)^{m}\wedge \beta^{n-m}+\displaystyle{\int}_{\lbrace u< v \rbrace}(-h)(\Delta \max \lbrace  u,v \rbrace)^{m}\wedge \beta^{n-m}\\
&\leq \displaystyle{\int}_{\Omega}(-h)(\Delta v)^{m} \wedge \beta^{n-m}+\displaystyle{\int}_{\lbrace u< v\rbrace}(-h)(\Delta v)^{m}\wedge \beta^{n-m}\\
&=\displaystyle{\int}_{\lbrace u> v\rbrace}(-h)(\Delta v)^{m}\wedge \beta^{n-m}
\end{array}$$
Let $h \searrow -1$ to get the desired result.
\end{enumerate}

\end{proof}
\subsection{The projection theorem}
\begin{defi}
Let $u:\Omega \longrightarrow \mathbb{R}\cup \lbrace -\infty \rbrace $ be an upper semi-continuous function. We define the projection of $u$ on $\mathcal{E}_{m}^{1}(\Omega) $ by
$$P(u)=\sup \lbrace  v \in \mathcal{E}_{m}^{1}(\Omega)\;:\;v\leq u \rbrace.$$
\end{defi}
\begin{lem}\label{lem5}
Let $u:\Omega \longrightarrow \mathbb{R}$ be a continuous function, suppose that there exists $w \in  \mathcal{E}_{m}^{1}(\Omega) $ such that $w\leq u$. Then $\displaystyle{\int}_{\lbrace P(u)<u\rbrace}(\Delta P(u))^{m}\wedge \beta^{n-m}=0.$
\end{lem}
\begin{proof}
Without loss of generality we can assume that $w$ is bounded. From Choquet’s lemma, there exists
an increasing sequence $(u_{j}) \subset \mathcal{E}_{m}^{1}(\Omega)\cap L^{\infty}(\Omega)$  such that
$(\lim_{j} u_{j})^{*}=P(u).$

Let $\tilde{q} \in \lbrace P(u)< u \rbrace .$ Since $u$ is continuous, there exist $\varepsilon > 0, r > 0$ such that
$$P(u)(q) < u(\tilde{q}) -\varepsilon 	 < u(q),\;\; \forall q \in B = B(\tilde{q}, r).$$
For  fixed $j$, by approximating $u_{j}\vert_{\partial B} $ from above by a sequence of continuous functions on $\partial B$ and by
using \cite[Theorem 3.1]{Ba}, we can find a function $\varphi_{j} \in \mathcal{QSH}_{m}(B)  $ such that $\varphi_{j} = u_{j}$ on $\partial B$ and $ (\Delta \varphi _{j})^{m}\wedge \beta^{n-m}=0$
in $B$. The comparison principle gives us that $\varphi_{j} \geq u_{j}$ in $B$. The function $\psi_{j}$ defined by $\psi_{j} = \varphi_{j}$ in $B$ and
$\psi_{j} = u_{j}$ in $\Omega \setminus B,$ belongs to $\mathcal{E}_{m}^{1}(\Omega)\cap L^{\infty}(\Omega).$ For each $q \in \partial B$ we have $\varphi_{j}(q) = u_{j}(q) \leq P(u)(q) \leq u(\tilde{q})-\varepsilon .$
It then follows that $\varphi_{j} \leq u(\tilde{q}) -\varepsilon$ in $B$ since $u(\tilde{q}) -\varepsilon $ is a constant and $\varphi_{j}\in \mathcal{QSH}_{m}(\Omega)$. Hence, $ u_{j} \leq \psi_{j} \leq u$
in $\Omega$. This implies that
$(\lim \psi_{j})^{*} = P(u)$.
It follows from Lemma \ref{lem2} that $(\Delta \psi_{j})^{m}\wedge \beta^{n-m} \rightharpoonup (\Delta P(u))^{m}\wedge \beta^{n-m}$. Therefore,
$(\Delta P(u))^{m}\wedge \beta^{n-m})(B) \leq \lim_{j\longrightarrow +\infty}\inf((\Delta \psi_{j})^{m}\wedge \beta^{n-m})(B)=0.$
from which the result follows.
\end{proof}
\begin{lem}\label{lem6}
Let  $u, v \in \mathcal{E}_{m}^{1}(\Omega)$ with $v$ is continuous. We define for $t < 0,$
$$h_{t}=\dfrac{P(u+tv)-tv-u}{t}.$$ Then, for each $0\leq k\leq m,$
$$\lim_{t \longrightarrow 0^{-}}\int_{\Omega}h_{t}(\Delta u)^{k}\wedge (\Delta P(u+tv))^{m-k}\wedge \beta^{n-m}=0. $$
In particular,
$$\lim_{t\longrightarrow 0^{-}}\int_{\Omega}\dfrac{P(u+tv)-u}{t}(\Delta u)^{k}\wedge (\Delta P(u+tv))^{m-k}\wedge \beta^{n-m}=\int_{\Omega}v(\Delta u)^{m}\wedge \beta^{n-m}. $$
\end{lem}
\begin{proof}
An easy computation shows that $h_{t}$ is decreasing in $t$ and $0 \leq h_{t} \leq -v.$ For each fixed $s < 0$ we have
$$\begin{array}{ll}
\displaystyle{\lim_{t \longrightarrow 0^{-}}\int_{\Omega}h_{t}(\Delta u)^{k}\wedge (\Delta P(u+tv))^{m-k}\wedge \beta^{n-m}}
&\leq \displaystyle{ \lim_{t \longrightarrow 0^{-}}\int}_{\Omega}h_{s}(\Delta u)^{k}\wedge (\Delta P(u+tv))^{m-k}\wedge \beta^{n-m}\\
&=\displaystyle{\int_{\Omega}h_{s}(\Delta u)^{k}\wedge (\Delta P(u))^{m-k}\wedge \beta^{n-m}}\\
&= \displaystyle\int_{\Omega}h_{s}(\Delta u)^{m}\wedge \beta^{n-m}\\ &=\displaystyle \int_{\lbrace P(u+sv)-sv < u \rbrace} h_{s}(\Delta u)^{m}\wedge \beta^{n-m}\\
& \leq \displaystyle{ \int_{\lbrace P(u+sv)-sv < u \rbrace}(-v)(\Delta u)^{m}\wedge \beta^{n-m}}.
\end{array}$$
Let $ u_{k}\in \mathcal{E}_{m}^{0}(\Omega)\cap C(\Omega) $ be a decreasing sequence  which tends to $u$ such that
$$\int_{\lbrace P(u+sv)-sv < u \rbrace}(-v)(\Delta u)^{m}\wedge \beta^{n-m} \leq A \int_{\lbrace P(u_{k}+sv)-sv < u \rbrace}(-v)(\Delta u)^{m}\wedge \beta^{n-m}.$$
Taking into account Theorem \ref{th3}  and  proposition \ref{pro15} we can conclude that
$$\begin{array}{ll}
\displaystyle{\int_{\lbrace P(u_{k}+sv)-sv < u \rbrace}h_{s}(-v)(\Delta u)^{m}\wedge \beta^{n-m}}
&\leq \displaystyle{\int_{\lbrace P(u_{k}+sv)-sv < u \rbrace}(-v)(\Delta(P(u_{k}+sv)-sv  )^{m}\wedge \beta^{n-m}}\\
&\leq \displaystyle{\int_{\lbrace P(u_{k}+sv)-sv < u_{k} \rbrace}(-v)(\Delta(P(u_{k}+sv)-sv  )^{m}\wedge \beta^{n-m}}\\
&\leq -sM \longrightarrow 0, \mbox{ as } s\longrightarrow 0.
\end{array}$$

Where $ M$ is a positive constant which depends only on $m,\; \Vert v\Vert$ and $\displaystyle{\int_{\Omega}v(\Delta( u+v))^{m}\wedge \beta^{n-m}} .$ The second equality  follows
from the first one. Thus we complete the proof.
\end{proof}
\begin{thm}\label{th10}
Let  $u, v \in \mathcal{E}_{m}^{1}(\Omega),$ and $v$ is continuous. Then
$$\dfrac{d}{dt}\Big| _{t=0}E(P(u+tv))=(m+1)\int_{\Omega}(-v)(\Delta u)^{m}\wedge \beta^{n-m}. $$
\end{thm}
\begin{proof}
If $t > 0, P(u + tv) = u + tv.$ It is easy to see that
$$\dfrac{d}{dt}\Big| _{t=0^{+}}E(P(u+tv))=(m+1)\int_{\Omega}(-v)(\Delta u)^{m}\wedge \beta^{n-m}. $$
For $t< 0$, observing that $P(u+tv), \;u \in \mathcal{E}_{m}^{1}(\Omega)$, we can integrate by parts to have
$$\begin{array}{ll}
&\displaystyle{\dfrac{1}{t}\Big(\int_{\Omega}-P(u+tv)(\Delta P(u+tv))^{m}\wedge  \beta^{n-m}-\int_{\Omega}(-u)(\Delta u)^{m}\wedge  \beta^{n-m}\Big)}\\
&=\displaystyle{\sum_{k=0}^{m}\int_{\Omega}\dfrac{u-P(u+tv)}{t}(\Delta u)^{k}\wedge(\Delta P(u+tv))^{m-k}\wedge   \beta^{n-m}}.
\end{array}$$
It suffices to apply Lemma \ref{lem6}.
\end{proof}
\subsection{The quaternionic Hessian equation}
In this section, we introduce the variational method to solve the quaternionic Hessian equation on finite energy classes of Cegrell type $$(\Delta u)^{m}\wedge \beta^{n-m}=\mu,$$
where $\mu$ is a positive Radon measure.
 The idea is to minimize the energy functional on a compact subset of $m$-sh functions.
\begin{lem}\label{lem7}
Assume that $\mu$ is a positive Radon measure such that $\mathcal{F}_{\mu}$ is proper and lower semi-continuous on $\mathcal{E}_{m}^{1}(\Omega)$. Then, there exists $\varphi \in \mathcal{E}_{m}^{1}(\Omega)$ such that $\mathcal{F}_{\mu}(\varphi)=\inf_{\psi\in \mathcal{E}_{m}^{1}(\Omega)}\mathcal{F}_{\mu}(\psi).$
\end{lem}
\begin{proof}
As in the proof of \cite[lemma 4.12]{Lu2}.
Let $(\varphi_{j} ) \in \mathcal{E}_{m}^{1}(\Omega)$  be such that$$ \displaystyle \lim_{j}\mathcal{F}_{\mu}(\varphi_{j})=\inf_{\psi\in \mathcal{E}_{m}^{1}(\Omega)}\mathcal{F}_{\mu}(\psi) \leq 0. $$

From the properness of the functional $\mathcal{F}_{\mu}$ , we obtain $\sup_{j}E(\varphi_{j}) < + \infty$ . It follows that the sequence $ (\varphi_{j} )$
forms a compact subset of $\mathcal{E}_{m}^{1}(\Omega)$ . Hence there exists a subsequence $(\varphi_{j_{k}})$
converging to $\varphi$ in $L_{loc}^1(\Omega) $ Since $\mathcal{F}_{\mu}$  is lower semi-continuous we have$$\displaystyle \liminf_{j\longrightarrow +\infty}\mathcal{F}_{\mu}(\varphi_{j_{k}}) \geq \mathcal{F}_{\mu}(\varphi) $$
We then deduce that $\varphi$  is a minimum point of $\mathcal{F}_{ \mu}$  on $\mathcal{E}_{m}^{1}(\Omega)$.
\end{proof}
\begin{thm}\label{th11}
Suppose that $\varphi \in \mathcal{E}_{m}^{1}(\Omega)$ and $\mu \in \mathcal{M}_{1}$. Then
$$(\Delta \varphi)^{m}\wedge \beta^{n-m}=\mu\Longleftrightarrow \mathcal{F}_{\mu}(\varphi)=\inf_{\psi\in \mathcal{E}_{m}^{1}(\Omega)}\mathcal{F}_{\mu}(\psi).$$
\end{thm}
\begin{proof}
First assume that $(\Delta \varphi)^{m}\wedge \beta^{n-m}=\mu$ and let $\psi \in \mathcal{E}_{m}^{1}(\Omega)$. From Theorem \ref{th6} and Young's inequality follow
$$\begin{array}{ll}
\displaystyle{\int}_{\Omega}(-\psi)(\Delta \varphi)^{m}\wedge \beta^{n-m}& \leq E(\psi)^{\frac{1}{m+1}}.E(\varphi)^{\frac{m}{m+1}}\\
&\leq \dfrac{1}{m+1}E(\psi)+\dfrac{m}{m+1}E(\varphi).
\end{array}$$
Then $\mathcal{F}_{\mu}(\psi)\geq \mathcal{F}_{\mu}(\varphi).$ Thus $\mathcal{F}_{\mu}(\varphi)=\inf_{\psi\in \mathcal{E}_{m}^{1}(\Omega)}\mathcal{F}_{\mu}(\psi).$

Now, assume that $\mathcal{F}_{\mu}$ is minimized on $\mathcal{E}_{m}^{1}(\Omega)$ at $\varphi$. let $\psi \in \mathcal{E}_{m}^{1}(\Omega)\cap C(\Omega)$ and define
$$f(t)=\dfrac{1}{m+1}E(P(\varphi +t \psi))+L_{\mu}(\varphi +t \psi),\;\;t\in \mathbb{R}.$$
Using Theorem \ref{th10} we get $$f^{'}(0)=\int_{\Omega}(-\psi)(\Delta u)^{m}\wedge \beta^{n-m}+L_{\mu}( \psi).$$
Since $P(\varphi +t \psi)\leq \varphi +t \psi $ and $P(\varphi +t \psi) \in \mathcal{E}_{m}^{1}(\Omega), $ then $$f(t)\geq \mathcal{F}_{\mu}(P(\varphi + t \psi))\geq \mathcal{F}_{\mu}(\varphi)=f(0),\;\forall t \in \mathbb{R}.$$
It follows that $f$ attains its minimum at $t=0$. Thus $f^{'}(0)=0.$ Therefore $$\int_{\Omega}(\psi)(\Delta \varphi)^{m}\wedge \beta^{n-m}=\int_{\Omega}\psi d\mu.$$ for an arbitrarily test function $\psi$ which implies the result.
\end{proof}
\begin{lem}\label{lem8}
Let  $u, v \in \mathcal{E}_{m}^{1}(\Omega)$ such that $(\Delta u)^{m}\wedge \beta^{n-m}\geq (\Delta v)^{m}\wedge \beta^{n-m}.$ Then $u\leq v$ in $\Omega.$
\end{lem}
\begin{proof}
 By the absurd, suppose that there exists $q_{0} \in \Omega$ such that $v(q_{0})< u(q_{0}).$ Let $\varphi$ an exhaustion function of $\Omega$, such that $\varphi(q_{0})< - \epsilon r^{2}$
 for each $q \in B(q_{0},r) \cap \Omega,\;\;r > 0$ for a fixed $\epsilon >0$ and smaller enough.
Define $\psi(q):= \max \{ \varphi(q),\; \epsilon (\vert q-q_{0} \vert ^{2} -r^{2} )\}$. Then $\Psi$ is a continuous exhaustion function in $\Omega$ such that $(\Delta \psi)^{m}\wedge \beta^{n-m}\geq  \epsilon^{m}\beta^{n-m},$ near $q_{0}.$
Choose $\delta > 0$ so smaller such that $v(q_{0}) < u(q_{0}) + \delta \psi(q_{0}),$ then the Lebesgue measure of the set $\mathcal{U}:= \{ q \in \Omega:\; v(q) + \delta \psi(q) \} \cap B(q_{0},R)$ is strictly positive for $R > 0.$ Then $ \displaystyle \int_{ \mathcal{U}}(\Delta \psi)^{m}\wedge \beta^{n-m} > 0.$
From proposition \ref{pro15} (6) it follows that $$\begin{array}{ll}

   \displaystyle \int_{ \mathcal{U}}(\Delta v+ \delta \psi)^{m}\wedge \beta^{n-m}   & \leq \displaystyle \int_{ \mathcal{U}}(\Delta u)^{m}\wedge \beta^{n-m}    \\
     & \leq \displaystyle \int_{ \mathcal{U}}(\Delta v)^{m}\wedge \beta^{n-m}.
\end{array} $$
Hence $$\displaystyle \int_{ \mathcal{U}}(\Delta v)^{m}\wedge \beta^{n-m} + \delta^{m} \displaystyle \int_{ \mathcal{U}}(\Delta \psi)^{m}\wedge \beta^{n-m} \leq \displaystyle \int_{ \mathcal{U}}(\Delta v)^{m}\wedge \beta^{n-m} .  $$
which is a contradiction.
\end{proof}
\begin{lem}\label{lem9}
Let $\mu$ be a positive Radon measure in $\Omega$ does not charge $m$-polar sets such that $\mu(\Omega) < + \infty$.
 Let $(u_{j})$ be a sequence in $\mathcal{QSH}_{m}^{-}(\Omega)$ which converges in $L_{loc}^{1}$  to
 $u \in \mathcal{QSH}_{m}(\Omega).$ If $\displaystyle \sup_{j} \int_{\Omega}(-u_{j})^{2} d \mu < + \infty$
then $\displaystyle \int_{\Omega}u_{j} d \mu \longrightarrow \int_{\Omega} u d \mu .$
\end{lem}
\begin{proof}
Since $\displaystyle \sup_{j} \int_{  \Omega}(-u_{j})^{2} d \mu < + \infty$, by Banach-Saks Theorem there exists a sub-sequence $(u_{j})$ such that $\varphi_{N}:= \displaystyle \dfrac{1}{N} \sum_{j=1}^{N} u_{j} $ converges in $L^{2} ( \mu)$ and $\mu$-almost everywhere to $\varphi.$
We have also $\varphi_{N}
 \rightarrow u$ in $L_{loc}^{1}.$ For each $j \in \mathbb{N}$, Denote by $\psi_{j}:=( \sup_{k \geq j} \varphi_{k})^{\ast}$. Then   $\psi_{j}$ decreases to $u$ in $\Omega$. Since $\mu$ does not charge the $m$-polar set $\{(\sup_{k\geq j}\varphi_{k})^{\ast} > \sup_{k\geq j}\varphi_{k} \}$. Then we conclude that $\psi_{j}:=\sup_{k\geq j}\varphi_{k}$ $\mu$-almost everywhere. Thus $\psi$ converges to $\varphi$ $\mu$- almost everywhere and $u=\varphi$ $\mu$-almost everywhere. This yields $$ \displaystyle \lim_{j} \int_{\Omega} u_{j} d \mu=\lim_{j} \int_{\Omega} \varphi_{j} d \mu=\lim_{j} \int_{\Omega} u d \mu$$

\end{proof}

\begin{thm}\label{th12}
Suppose that  $\mu \in \mathcal{M}_{1}$. Then there exists a unique $u \in \mathcal{E}_{m}^{1}(\Omega)$ such that
$(\Delta u)^{m}\wedge \beta^{n-m}=\mu.$
\end{thm}
\begin{proof}
The uniqueness follows from Lemma \ref{lem8}.
We prove the existence in two steps.

$\textbf{Step 1}$: If $\mu$ has compact support $K \Subset \Omega$, let $h_{K} = h_{m,K,\Omega}^{*}$ be the regularized relatively $m$-extremal function of $K$
 with respect to $\Omega$ and  set
$$\mathcal{M}=\Big \lbrace \nu>0\;:\;supp\; \nu \subset K,\;\;\int_{\Omega}(-\varphi)^{2}d \nu\leq C E(\varphi)^{\frac{2}{m+1}},\mbox{ for every } \varphi \in \mathcal{E}_{m}^{1}(\Omega) \Big\rbrace,$$
where $C$ is a fixed constant such that $ C > 2E(h_{K})^{\frac{m-1}{m+1}}$ . For each compact $L \subset K,$ we have $h_{K} \leq h_{L}$. Then
 $E(h_{L}) \leq E(h_{K})$. Therefore, for every $\varphi \in \mathcal{E}_{m}^{1}(\Omega)$, we have by (\ref{eq26})
$$\begin{array}{ll}
\displaystyle{\int}_{\Omega}(-\varphi)^{2}(\Delta h_{L})^{m}\wedge \beta^{n-m}& \leq 2 \Vert  h_{L}\Vert_{\Omega}\displaystyle{\int}_{\Omega}(-\varphi)\Delta \varphi \wedge (\Delta h_{L})^{m-1}\wedge \beta^{n-m}\\
&\leq 2 E(\varphi)^{\frac{2}{m+1}}.E (h_{L})^{\frac{m-1}{m+1} }\\ \displaystyle &  \leq 2 E(\varphi)^{\frac{2}{m+1}}.E(h_{K})^{\frac{m-1}{m+1}}\\ & \displaystyle< C E(\varphi)^{\frac{2}{m+1}}.
\end{array}$$
This implies that $(\Delta h_{L})^{m} \wedge \beta^{n-m} \in \mathcal{M}$ for every compact $L \subset K.$
Put $T = \sup \lbrace \nu (\Omega),\;\; \nu \in \mathcal{M} \rbrace$. We have $T <  + \infty$. In fact, since $\Omega$ is $m$-hyperconvex, there exists
$g \in \mathcal{QSH}_{m}^{-}(\Omega) \cap C(\overline{\Omega})$ such that $ g \leq -1 $ on $K \Subset \Omega$. For each $\nu \in \mathcal{M}$, we have
$ \nu(K)\leq \displaystyle{\int}_{\Omega}(-g)^{2}d \nu \leq C E(g)^{\frac{2}{m+1}},$
from which the result follows.

Fix $\nu_{0} \in \mathcal{M}$ such that $\nu_{0}(\Omega)> 0.$   Set
$$\mathcal{N}=\Big \lbrace \nu>0\;:\;supp\; \nu \subset K,\;\;\int_{\Omega}(-\varphi)^{2}d \nu\leq \Big( \dfrac{C}{T}+\dfrac{C}{\nu_{0}(\Omega)} \Big) E(\varphi)^{\frac{2}{m+1}},\mbox{ for every } \varphi \in \mathcal{E}_{m}^{1}(\Omega) \Big\rbrace,$$
Then, for each $\nu  \in \mathcal{M}$ and $\varphi \in \mathcal{E}_{m}^{1}(\Omega) ,$

$$\begin{array}{ll}
\displaystyle{\int}_{\Omega}(-\varphi)^{2}\dfrac{(T-\nu(\Omega))d \nu_{0}+\nu_{0}(\Omega)d \nu}{T \nu_{0}(\Omega)}&\leq \dfrac{T- \nu(\Omega)}{T \nu_{0}(\Omega)}\displaystyle{\int}_{\Omega}(- \varphi^{2})d \nu_{0}+\dfrac{1}{T}\displaystyle{\int}_{\Omega}(- \varphi^{2})d \nu \\
&\leq \Big(C\dfrac{T- \nu(\Omega)}{T \nu_{0}(\Omega)}+\dfrac{C}{T} \Big) E(\varphi)^{\frac{2}{m+1}} \\
&\leq \Big(\dfrac{C}{\nu_{0}(\Omega)}+\dfrac{C}{T} \Big) E( \varphi)^{\frac{2}{m+1}}.
\end{array}$$
From this we conclude that
$\dfrac{(T- \nu(\Omega))\nu_{0}+ \nu_{0}(\Omega)\nu}{T \nu_{0}(\Omega)} \in \mathcal{N}$, for every $\nu \in \mathcal{M}.$
Therefore $\mathcal{N}$ is nonempty convex and weakly compact in the space of probability measures.
 From a generalized Radon-Nykodim Theorem follows that there exists a positive measure $\nu \in \mathcal{N}$ and
a positive function $f \in L^{1}(\nu)$ such that $\mu= fd \nu+ \nu_{1}$, where $\nu_{1}$ is orthogonal to $\mathcal{N}.$
Since $(\Delta h_{L})^{m}\wedge \beta^{n-m}\in \mathcal{N}$ for each $L \Subset K,$
each measure orthogonal to $\mathcal{N}$ must be supported in some $m$-polar set. Since $\mu$ does not charge $m$-polar sets,
then we deduce that $\nu_{1} \equiv 0.$

 For each $j \in \mathbb{N}$ set $\mu_{j} = \min(f, j) \nu.$  From Lemma \ref{lem9} and Proposition \ref{pro15}, we deduce that $\mathcal{L}_{\mu_{j}}$ is  continuous on $\mathcal{E}_{m}^{1}(\Omega)$ and $\mathcal{F}_{\mu_{j}}$ is proper and lower semi-continuous.
 Therefore, by Lemma \ref{lem7}  and Theorem \ref{th11}, there exists $u_{j}\in \mathcal{E}_{m}^{1}(\Omega)$ such that
$(\Delta u_{j})^{m}\wedge \beta^{n-m} = \mu_{j}.$ It is clear from the comparison principle that $(u_{j})$ decreases to a function $u \in \mathcal{E}_{m}^{1}(\Omega)$ which
solves $(\Delta u)^{m}\wedge \beta^{n-m} = \mu.$

$\textbf{Step 2}$:
If  $\mu$ does not have compact support. Let $(K_{j})$ be an exhaustive sequence
of compact subsets of $\Omega$ and let $u_{j}\in \mathcal{E}_{m}^{1}(\Omega) $ such that $(\Delta u_{j})^{m}\wedge \beta^{n-m} = \mu_{j}$ where $\mu_{j} = \chi_{K_{j}}d \mu.$ We have
$(u_{j})$ decreases to $u \in \mathcal{QSH}_{m}^{-}(\Omega) .$ We will prove that $\sup_{j} E(u_{j}) < +\infty.$ Indeed, since $\mu \in \mathcal{M}_{1},$ then
$$ E(u_{j})= \int_{\Omega} (-u_{j})(\Delta u_{j})^{m}\wedge \beta^{n-m}= \int_{K_{j}} (-u_{j})d \mu \leq \int_{\Omega} (-u_{j})d \mu \leq A E(u_{j})^{\frac{1}{m+1}}.$$
This implies that $E(u_{j})$ is uniformly bounded, hence $u \in \mathcal{E}_{m}^{1}(\Omega)$ and the result follows.

\end{proof}
\begin{lem}\label{lem10}
Let  $\mu$ be a positive Radon measure satisfying $\mu(\Omega)< +\infty$, and $\mu\leq (\Delta \psi)^{m}\wedge \beta^{n-m}$, where $\psi$ is a bounded function in $ \mathcal{QSH}_{m}(\Omega).$
 Then there exists a unique function $\varphi \in \mathcal{E}_{m}^{0}(\Omega)$
such that $(\Delta \varphi)^{m}\wedge \beta^{n-m}=\mu.$
\end{lem}
\begin{proof}
Assume that $-1 \leq \psi \leq 0$. Let $h \in \mathcal{E}_{m}^{0}(\Omega)$ be the exhaustion function of $\Omega$. Let $h_{j}=\max\lbrace \psi, jh\rbrace$ and $A_{j}=\lbrace q \in \Omega\;:\;jh<-1\rbrace.$
 Note that $\chi_{A_{j}}\mu \in \mathcal{M}_{1}$, Theorem \ref{th12} implies that for
each $j$, there exists $\varphi_{j}\in \mathcal{E}_{m}^{1}(\Omega)$  such that $(\Delta \varphi_{j})^{m}\wedge\beta^{n-m}= \chi_{A_{j}}\mu$.
 Thus
$0 \geq \varphi_{j} \geq h_{j} \geq  \psi $ on $\Omega.$  By
Lemma \ref{lem8}, $\varphi_{j}$ decreases to some $\varphi \in \mathcal{E}_{m}^{0}(\Omega)$ and $\varphi$ satisfies $(\Delta \varphi)^{m}\wedge \beta^{n-m}=\mu. $
\end{proof}
\begin{pro}\label{pro16}
 If $u, v \in \mathcal{E}_{m}^{p}(\Omega), \;p > 1$, then
$$\int_{\lbrace u>v\rbrace}(\Delta u)^{m}\wedge \beta^{n-m}\leq \int_{\lbrace u>v\rbrace}(\Delta v)^{m}\wedge \beta^{n-m}.$$
\end{pro}
\begin{proof}
 Let $h \in \mathcal{E}_{m}^{0}(\Omega)\cap C(\Omega).$
First assume that $v$ is bounded and vanishes on the boundary. Let $ K_{j}$
be an exhaustion sequence of compact subsets of $\Omega.$ Using Lemma \ref{lem10}, there exists $v_{j}\in \mathcal{E}_{m}^{0}(\Omega)$ such that $(\Delta v_{j})^{m}\wedge \beta^{n-m}=\chi_{K_{j}}(\Delta v)^{m}\wedge \beta^{n-m}.$ Then, by lemma \ref{lem8} $v_{j}\downarrow v.$
 Now, from  $\displaystyle{\int}_{\Omega}(-h)(\Delta v_{j})^{m}\wedge \beta^{n-m}< +\infty$ and Corollary \ref{cor8} follow
$$ \displaystyle{\int}_{\lbrace u> v_{j}\rbrace}(-h)(\Delta u)^{m}\wedge \beta^{n-m}\leq \displaystyle{\int}_{\lbrace u> v_{j}\rbrace}(-h)(\Delta v_{j})^{m}\wedge \beta^{n-m}=\displaystyle{\int}_{\lbrace u> v_{j}\rbrace\cap K_{j}\rbrace}(-h)(\Delta v)^{m}\wedge \beta^{n-m}. $$
Letting $j \longrightarrow +\infty$ we get
$$ \displaystyle{\int}_{\lbrace u> v \rbrace}(-h)(\Delta u)^{m}\wedge \beta^{n-m}\leq \displaystyle{\int}_{\lbrace u> v\rbrace}(-h)(\Delta v)^{m}\wedge \beta^{n-m}.$$
It remains to remove the assumption on $v$ as in the proof of \cite[Theorem 5.2]{Lu2}.
\end{proof}
\begin{pro}\label{pro17}
Let  $\mu$ be a positive  measure in $\Omega$ which does not charge $m$-polar sets.
 Then, there exists  $\varphi \in \mathcal{E}_{m}^{0}(\Omega)$ and $0\leq f\in L_{loc}^{1}((\Delta \varphi)^{m}\wedge \beta^{n-m})$ such that $\mu =f((\Delta \varphi)^{m}\wedge \beta^{n-m}).$
\end{pro}
\begin{proof}
We first assume that $\mu$ has compact support. By applying Theorem \ref{th12}  we can find $u \in \mathcal{E}_{m}^{1}(\Omega)$  and
$0 \leq f \in L^{1}((\Delta u)^{m}\wedge \beta^{n-m})$ such that $\mu =f((\Delta u)^{m}\wedge \beta^{n-m})$, and $supp((\Delta u)^{m}\wedge \beta^{n-m})\Subset \Omega.$ Let
$\psi = (-u)^{-1} \in \mathcal{QSH}_{m}(\Omega)\cap L_{loc}^{\infty}(\Omega) .$
Then $(-u)^{-2m}((\Delta u)^{m}\wedge \beta^{n-m}) \leq (\Delta \psi)^{m} \wedge \beta^{n-m}.$ Since $(\Delta u)^{m}\wedge\beta^{n-m}$ has compact support in $\Omega$, we can modify $\psi$ in a neighborhood
of $\partial \Omega$ such that $\psi \in \mathcal{E}_{m}^{0}(\Omega) .$ It follows from Lemma \ref{lem10}  that
$$(-u)^{-2m}((\Delta u)^{m}\wedge \beta^{n-m}) =(\Delta \varphi)^{m}\wedge \beta^{n-m} ,\; \varphi \in \mathcal{E}_{m}^{0}(\Omega).$$
This implies that  $\mu = f(-u)^{2m}((\Delta \varphi)^{m}\wedge \beta^{n-m}).$

If  $\mu$ has compact support. Let $(K_{j})$ be an exhaustive sequence
of compact subsets of $\Omega$. From previous arguments there exist $u_{j} \in \mathcal{E}_{m}^{0}(\Omega)$ and $f_{j} \in L^{1}((\Delta u_{j})^{m} \wedge \beta^{n-m})$ such that
$\chi_{K_{j}} \mu = f_{j}((\Delta u_{j})^{m}\wedge \beta^{n-m}).$ Take a sequence of positive numbers $(t_{j})$ satisfying $\varphi = \displaystyle{\sum_{j=1}^{\infty}t_{j}u_{j}} \in \mathcal{E}_{m}^{0}(\Omega).$ The
measure $\mu$ is absolutely continuous with respect to $(\Delta \varphi)^{m} \wedge \beta^{n-m}.$ Thus
$\mu = g((\Delta u_{j})^{m}\wedge \beta^{n-m})$ and $g \in L_{loc}^{1}((\Delta \varphi)^{m}\wedge \beta^{n-m}).$
\end{proof}
\begin{thm}\label{th13}
Let  $\mu$ be a positive Radon  measure in $\Omega$ such that $ \mathcal{E}_{m}^{p}(\Omega)\subset L^{p}(\Omega,\mu),\;p\geq 1$.
 Then, there exists a unique  $\varphi \in \mathcal{E}_{m}^{p}(\Omega)$  such that $\mu =(\Delta \varphi)^{m}\wedge \beta^{n-m}.$
\end{thm}
\begin{proof}
The uniqueness follows from Proposition \ref{pro16}. Since $\mu$ does not charge $m$-polar sets,
by Proposition \ref{pro17} there exist $u \in \mathcal{E}_{m}^{0}(\Omega)$  and $0 \leq f \in L_{loc}^{1} ((\Delta \varphi)^{m}\wedge \beta^{n-m})$ such that $\mu = f ((\Delta \varphi)^{m}\wedge \beta^{n-m})$.
For each $j,$ let $\mu_{j}= \min(f, j)((\Delta \varphi)^{m}\wedge \beta^{n-m})$. By Lemma \ref{lem10}, we can find $\varphi_{j} \in \mathcal{E}_{m}^{0}(\Omega)$  such that
$(\Delta \varphi_{j})^{m}\wedge \beta^{n-m} = \mu_{j}.$ Since $\mu \in \mathcal{M}_{p},$ we have $\sup_{j} E_{p}(\varphi_{j}) < +\infty.$ It follows from Proposition \ref{pro16}
and definition of $\mathcal{E}_{m}^{p}(\Omega)$
 that $\varphi_{j}$ decreases to some $\varphi \in \mathcal{E}_{m}^{p}(\Omega)$ and $\varphi$ satisfies $(\Delta \varphi)^{m}\wedge \beta^{n-m}=\mu.$
\end{proof}
\begin{lem}\label{lem11}
Let $u,v \in \mathcal{E}_{m}^{p}(\Omega)$ and $p\geq 1$. Then there exist two sequences $(u_{j}),(v_{j}) \subset \mathcal{E}_{m}^{0}(\Omega)$ decreasing to $u,v$ respectively, such that
$$ \lim_{j\longrightarrow +\infty}\int_{\Omega}(-u_{j})^{p}(\Delta v_{j})^{m}\wedge\beta^{n-m}=\int_{\Omega}(-u)^{p}(\Delta v)^{m}\wedge\beta^{n-m}.$$
\end{lem}
\begin{proof}
Since $u \in \mathcal{E}_{m}^{p}(\Omega)$, there exists a sequence $(u_{j})\subset \mathcal{E}_{m}^{0}(\Omega)$ decreasing to $u$ such that $$\sup_{j}\int_{\Omega}(-u_{j})^{p}(\Delta u_{j})^{m}\wedge\beta^{n-m}<+\infty.$$
From Proposition \ref{pro17} and the fact that $(\Delta v)^{m}\wedge \beta^{n-m}$ does not charge $m$-polar sets, we can find $\psi  \in \mathcal{E}_{m}^{0}(\Omega)$ and $0\leq f \in L_{loc}^{1}((\Delta \psi)^{m}\wedge \beta^{n-m})$ such that  $(\Delta v)^{m}\wedge \beta^{n-m}=f((\Delta \psi)^{m}\wedge \beta^{n-m}).$ Then by lemma \ref{lem10}, there exists a sequence $(v_{j})\subset \mathcal{E}_{m}^{0}(\Omega)$ such that $(\Delta v_{j})^{m}\wedge \beta^{n-m}=\min(f,j)(\Delta \psi)^{m}\wedge \beta^{n-m}.$
Thus from the comparison principle follows that $(v_{j})$ decreases to some function $\varphi \in \mathcal{E}_{m}^{p}(\Omega)$ such that $(\Delta v)^{m}\wedge \beta^{n-m}=(\Delta \varphi)^{m}\wedge
 \beta^{n-m}.$ Hence, we have $v\equiv \varphi.$ Therefore,
 $$\lim_{j\longrightarrow +\infty}\int_{\Omega}(-u_{j})^{p}(\Delta v_{j})^{m}\wedge\beta^{n-m}=\lim_{j\longrightarrow +\infty}\int_{\Omega}(-u_{j})^{p}\min(f,j)(\Delta \psi)^{m}\wedge\beta^{n-m}=\int_{\Omega}(-u)^{p}(\Delta v)^{m}\wedge\beta^{n-m}.$$
\end{proof}
\begin{proof}[\textbf{Proof of Theorem} \ref{th00}]
Assume that $\mu=(\Delta \varphi)^{m}\wedge \beta^{n-m} $ with $\varphi\in \mathcal{E}_{m}^{p}(\Omega).$
and $\psi $ is an other function in $\mathcal{E}_{m}^{p}(\Omega).$ By lemma \ref{lem11} we can find two sequences $(\varphi_{j}),(\psi_{j}) \subset \mathcal{E}_{m}^{0}(\Omega)$ decreasing to $\varphi,\psi$ respectively such that $$\sup_{j}\int_{\Omega}(-\varphi_{j})^{p}(\Delta \varphi_{j})^{m}\wedge\beta^{n-m}<+\infty \mbox{ and } \sup_{j}\int_{\Omega}(-\psi_{j})^{p}(\Delta \psi_{j})^{m}\wedge\beta^{n-m}<+\infty.$$ From Theorem \ref{th6} it follows that $$\lim_{j\longrightarrow +\infty}\int_{\Omega}(-\psi_{j})^{p}(\Delta \varphi_{j})^{m}\wedge\beta^{n-m}=\int_{\Omega}(-\psi)^{p}(\Delta \varphi)^{m}\wedge\beta^{n-m}.$$
Then we get $\psi \in L^{p}(\Omega,\mu)$. It suffices to apply Theorem \ref{th13} to get the result.
\end{proof}

\end{document}